\documentclass[a4paper, reqno, 11pt, notitlepage]{amsart}
    \usepackage[utf8]{inputenc} 
	\usepackage[a4paper,margin=1.2in]{geometry}
     \usepackage[american]{babel}
    \usepackage{amsmath,amssymb,amsthm,graphicx,color,eucal,stmaryrd,scalerel}
    \usepackage[shortlabels]{enumitem} 
    \usepackage[all,cmtip]{xy}
    \usepackage[T1]{fontenc}
    
    \usepackage{hyperref} 
    \definecolor{darkblue}{rgb}{0,0,.85} 
    \definecolor{darkred}{rgb}{0.84,0,0}
    \hypersetup{colorlinks = true,
            linkcolor = darkblue,
            urlcolor  = darkblue,
            citecolor = darkred,
            anchorcolor = darkblue}

    \SetLabelAlign{LeftAlignWithIndent}{\hspace*{2.0ex}\makebox[1.5em][l]{#1}}
    
    \makeatletter
    \def\paragraph{\@startsection{paragraph}{4}%
    \z@\z@{-\fontdimen2\font}%
    {\normalfont\bfseries}}
    \makeatother
    
    \numberwithin{equation}{subsection}
        	
    \newdir{ >}{{}*!/-5pt/@{>}}
    \SelectTips{cm}{10}

    \newtheorem{lem}{Lemma}[subsection]
    \newtheorem{cor}[lem]{Corollary}
    \newtheorem{thm}[lem]{Theorem}
    \newtheorem{prop}[lem]{Proposition}
    \newtheorem*{thmi}{Theorem} 
    \newtheorem*{propi}{Proposition} 

    \theoremstyle{definition}
    \newtheorem{definition}[lem]{Definition}
    \newtheorem{construction}[lem]{Construction}
    \newtheorem{rem}[lem]{Remark}
    \newtheorem{example}[lem]{Example}
    \newtheorem*{definitioni}{Definition} 
    
    \theoremstyle{plain}
    \newtheorem{lemA}{Lemma}[section]
    \newtheorem{corA}[lemA]{Corollary}
    \newtheorem{propA}[lemA]{Proposition}
    \theoremstyle{definition}
    \newtheorem{definitionA}[lemA]{Definition}
    
    \newcommand{\mf}[1]{\mathfrak{#1}}
    
    \newcommand{\mc}[1]{\mathcal{#1}}
    \newcommand{\mb}[1]{\mathbf{#1}}
    
    \newcommand{\ov}[1]{\overline{#1}}

    \newcommand{\op}{\operatorname}
    \DeclareMathOperator*{\colim}{colim}
    \DeclareMathOperator{\Hom}{Hom}
    \DeclareMathOperator{\Aut}{Aut}
    \DeclareMathOperator{\Isom}{Isom}
    \DeclareMathOperator{\im}{im}
    \DeclareMathOperator{\Spf}{Spf}
    \DeclareMathOperator{\Spec}{Spec}
    \DeclareMathOperator{\Spa}{Spa}
    \DeclareMathOperator{\interior}{int}
    
    \newcommand{\oc}{\mathrm{oc}}
    \newcommand{\et}{\mathrm{\acute{e}t}}
    \newcommand{\adm}{\mathrm{adm}}
    \newcommand{\mx}{\mathrm{max}}
    \newcommand{\proet}{\mathrm{pro\acute{e}t}}
    
    \newcommand{\alg}{\mathrm{alg}}
    \newcommand{\sep}{\mathrm{sep}}
    \newcommand{\cnts}{\mathrm{cnts}}

    \newcommand{\cat}[1]{\operatorname{\mathbf{#1}}} 
    \newcommand{\Set}{\cat{Set}}
    \newcommand{\GSet}[1]{{#1\text{-}\cat{Set}}}
    \newcommand{\Cov}{\cat{Cov}}
    
    \newcommand{\Et}{\cat{\acute{E}t}}
    \newcommand{\FEt}{\cat{F\acute{E}t}}
    \newcommand{\UFEt}{\cat{UF\acute{E}t}}

    \newcommand{\h}{\mathcal{O}}

    \newcommand{\cO}{\mathcal{O}}
    \newcommand*\isomto{%
        \xrightarrow{\raisebox{-0.2 em}{\smash{\ensuremath{\sim}}}}%
    }
    \newcommand{\stacks}[1]{\cite[\href{https://stacks.math.columbia.edu/tag/#1}{Tag~#1}]{StacksProject}}
	
	\newcounter{steps}
    \newcommand{\beginsteps}{\setcounter{steps}{0}}
    \newcommand{\step}{\stepcounter{steps} \medskip \noindent {\it Step \arabic{steps}.} }
    
    \title{Geometric arcs and fundamental groups of rigid spaces} 
    \date{\today}
    
    \author{Piotr Achinger}
    \address{Institute of Mathematics of the Polish Academy of Sciences \newline \indent ul.\ Śniadeckich 8, 00-656 Warsaw, Poland}
    \email{pachinger@impan.pl}

    \author{Marcin Lara}
    \address{Institute of Mathematics of the Polish Academy of Sciences \newline \indent ul.\ Śniadeckich 8, 00-656 Warsaw, Poland}
    \email{marcin.lara@impan.pl}

    \author{Alex Youcis}
    \address{\begin{itemize} 
    \item[(1)] Institute of Mathematics of the Polish Academy of Sciences \newline \indent ul.\ Śniadeckich 8, 00-656 Warsaw, Poland
    \item[(2)] Graduate School of Mathematical Sciences, The University of Tokyo,
    3-8-1 Komaba, Meguro-ku, Tokyo, 153-8914, Japan
    \end{itemize}}
     \email{ayoucis@ms.u-tokyo.ac.jp}

\begin{document}

\begin{abstract}
We develop the notion of a \emph{geometric covering} of a rigid space $X$, which yields a larger class of covering spaces than that studied previously by de~Jong. Geometric coverings are closed under disjoint unions and are \'etale local on $X$. If $X$ is connected, its geometric coverings form a tame infinite Galois category, and hence are classified by a topological group. The definition is based on the property of lifting of ``geometric arcs'', and is meant to be an analogue of the notion developed for schemes by Bhatt and Scholze.
\end{abstract}

\maketitle

\section{Introduction}

While the geometry of non-archimedean analytic spaces has become relatively robust in the 60 years since its inception, a definitive theory of covering spaces has remained elusive. Ideologically, the main reason for the difficulty is that non-archimedean analytic spaces are not locally simply connected in any meaningful way. Here we mean `simply connected' in the rigid-geometric sense, not from a~topological perspective (work of Berkovich \cite{BerkovichContractible} shows that smooth $p$-adic Berkovich spaces are locally contractible as topological spaces). 

An important step in this direction, prompted by the theory of $p$-adic period mappings, was taken by de~Jong in \cite{deJongFundamental}. Following Berkovich \cite[6.3.4 ii)]{BerkovichEtale}, he combined the notions of topological covering and finite \'etale map by considering morphisms which locally in the Berkovich topology become the disjoint union of finite \'etale coverings. By doing so, he was able to develop a reasonable notion of the fundamental group of a rigid space.  Unfortunately, de~Jong's theory of covering spaces lacks some of the properties one would expect such as being closed under compositions and disjoint unions. Even more serious, the property of being a covering space in his sense is not local in the admissible topology on the target (see \cite[Theorem~1]{ALY1P2}).

Similar issues arise in topology when one considers covering spaces of topological spaces which aren't locally simply connected. An alternative perspective on covering spaces, present in the recent work of Brazas \cite{Brazas}, is to overcome this difficulty by converting one of the key properties of covering spaces into a definition --- the ability to lift paths. An analogue of this approach is also hidden in the work of Bhatt and Scholze on the pro-\'etale fundamental group in \cite{BhattScholze}, where the corresponding notion of a covering space is called a \emph{geometric covering}. Namely, if one defines a `geometric path' between two geometric points of a scheme $X$ as a sequence of specializations and generalizations of points in $X$ where subsequent points are `connected' by a strictly Henselian local ring, then geometric coverings are precisely \'etale maps for which one has unique lifting of all geometric paths.

The goal of this article is the development of a good theory of covering spaces in terms of path lifting, thus continuing the above pattern. It enlarges de Jong’s category\footnote{This is in contrast with the work of Andr\'e \cite{AndreLectures} and Lepage \cite{Lepage} on the tempered fundamental group, which aims at making the de~Jong fundamental group smaller and therefore more manageable.} and is based on the notion of a geometric arc, to be explained (along with the shift from paths to arcs) in more detail below.

\begin{definitioni}[See Definition~\ref{def:geom-cov}] \label{intro-def:geom-cov}
    Let $X$ be an adic space locally of finite type over a~non-archimedean field $K$. An \'etale morphism of adic spaces $Y\to X$ is a \emph{geometric covering} if it is partially proper and if the following condition holds:
    \begin{quote} 
        For every test curve $C\to X$, the map $Y_C\to C$ of adic curves satisfies unique lifting of geometric arcs.
    \end{quote}
    We denote by $\Cov_X$ the category of geometric coverings of $X$.
\end{definitioni}

In addition to our definition being geometrically intuitive, we show in \S\ref{s:geom-cov} that the notion of geometric covering is closed under composition, closed under disjoint unions, and is \'etale local on the target. Moreover, every finite \'etale covering is a geometric covering, and thus combining these results we see that every covering space in the sense of de Jong is a~geometric covering.\footnote{Our definition of geometric coverings matches quite nicely to that of a semi-covering as defined in \cite{Brazas}: the condition is the replacement for the unique lifting of paths, and \'etale and partially proper constitute the replacement for local homeomorphism (cf.\@ Proposition \ref{pp etale local structure}).}

To explain `geometric arcs' it is useful to note that, unlike schemes, for a geometric point $\ov{x}$ of a rigid space $X$, the \'etale localization $X_{(\ov{x})}$ (i.e.\@ the inverse limit of pointed \'etale neighborhoods) is often nothing more than a point (see \cite[Proposition 2.5.13 i)]{Huberbook}) and thus is not large enough to emulate the notion of `path' in algebraic geometry. But, also unlike schemes, rigid spaces $X$ (or more precisely their associated Berkovich space $[X]$) have an abundance of topological arcs, e.g.\ $[X]$ is arc connected if $X$ is connected. Our notion of \emph{geometric arcs} is a synthesis of the algebro-geometric and topological notions of paths: a geometric arc $\ov\gamma$ in a rigid space $X$ is an arc $\gamma$ in $[X]$ together with a choice of a geometric point above each of its points and a compatible family of \'etale paths along $\gamma$ (see Definition~\ref{def:geom-int}). 

If $X$ is one-dimensional, every two points of $X$ are connected by a geometric arc. This statement is significantly more subtle than the analogue from algebraic geometry (see Section~\ref{s:geom-arcs}). In general, extending a theorem of de~Jong \cite[Proposition~6.1.1]{deJongCrystal}, we show in Appendix~\ref{curve-connectedness appendix} that every two points of a connected $X$ can be connected by a sequence of \emph{test curves} (i.e.\@ one-dimensional rigid spaces over some extension of $K$). Combining these two statements we see that $X$ is ``geometric path connected''. We can then more precisely state the unique lifting property for a map $Y\to X$ as in Definition~\ref{intro-def:geom-cov} by saying that for every geometric arc $\ov\gamma$ in a test curve $C$ and every lifting $\ov y$ of the geometric left endpoint $\ov x$ of $\ov\gamma$ to $Y_C$, there exists a unique lifting of $\ov\gamma$ to a geometric arc in $Y_C$ with geometric left endpoint $\ov y$. 

Using this ``geometric path connectedness'', the path lifting property of our geometric coverings gives a concrete way of constructing an isomorphism of fiber functors $F_{\ov{x}_0}\simeq F_{\ov{x}_1}$ for any two geometric points of $X$ in the same connected component. This is an important ingredient in the proof of the following theorem, which is the main result of our paper.

\begin{thmi}[{See Theorem~\ref{main tameness result}}]
    For a connected adic space $X$ locally of finite type over a~non-archimedean field $K$ and a geometric point $\ov x$ of $X$, the pair $(\Cov_X, F_{\ov x})$ forms a~tame infinite Galois category in the sense of \cite[\S7]{BhattScholze}.
\end{thmi}

In words this means that the category $\Cov_X$ of geometric coverings has enough structure to support a notion of Galois theory. From the general yoga of tame infinite Galois categories one thus obtains from the pair $(\Cov_X,F_{\ov{x}})$ a Noohi topological group $\pi_1^{\rm ga}(X, \ov x)$, which we call the \emph{geometric arc fundamental group}\footnote{Given the analogies with \cite{BhattScholze}, it might seem reasonable to call this group the pro-\'etale fundamental group. However, we do not know if there is a suitable `pro-\'etale topology' whose local systems are $\Cov_X$.}, as well as an equivalence of categories 
\[ 
    F_{\ov x}\colon \Cov_X \isomto \GSet{\pi_1^{\rm ga}(X, \ov x)}.
\]
An analogue of this result was proven, in different language, by de Jong for the category of disjoint unions of his coverings (see \cite[Remark~7.4.11]{BhattScholze}). His proof, however, does not generalize to our situation. Moreover, our main theorem directly implies tameness of the natural generalizations of de Jong's category (see below).

The proofs of several key facts about our geometric coverings, most notably \'etale descent, rely on a more down-to-earth characterization in terms of a topological condition called the \emph{arcwise valuative criterion (AVC)}. To motivate it, note that unlike in the case of locally Noetherian schemes as in \cite{BhattScholze}, partially proper \'etale maps do not form a~well-behaved notion of a `covering space' (e.g.\ the inclusion of the open unit disk into the closed unit disk is partially proper and \'etale). Indeed, partial properness, which is defined using a valuative criterion, allows one to lift specializations, but on a connected rigid $K$-space $X$ most pairs of points cannot be connected by a zigzag of specializations. In fact, this is possible if and only if their images under the \emph{separation map}
\[ 
    \sep_X \colon X\to [X]
\]
in the corresponding separated quotient (Berkovich space) $[X]=X^{\rm Berk}$ are the same. Note that the fibers of $\sep_X$ are spaces of valuations (Riemann--Zariski spaces). The AVC for map of rigid $K$-spaces $Y\to X$ (see Definition~\ref{def:AVC}) is a lifting condition with respect to squares of the form
\[ 
    \xymatrix{
        [0,1) \ar[r] \ar[d] & [Y] \ar[d] \\
        [0,1] \ar[r]\ar@{.>}[ur]|-\exists & [X]
    }
\]
where $[0,1]\to [X]$ is a parametrized arc in the Berkovich space of $X$. A partially proper open immersion $U\to X$ with $X$ connected fails to satisfy this, and indeed the AVC complements well the partial properness condition: the valuative criterion included in the definition of a partially proper map works along the fibers of the separation map $\sep_X\colon X\to [X]$, whereas AVC works on the target $[X]$, and so the two jointly give rise to a situation much more representative of the valuative criterion from algebraic geometry (see Figure \ref{fig:val-criterion} for a graphical representation of this idea). The link between geometric coverings and the AVC is then the following.

\begin{propi}[{Proposition~\ref{geometric covering equiv}}] 
    Let $Y\to X$ be an \'etale and partially proper morphism of rigid spaces. Then, $Y\to X$ is a geometric covering if and only if for all test curves $C\to X$ the map $[Y_C]\to[C]$ satisfies AVC. 
\end{propi}

\begin{figure}[ht!]
    \centering
    \includegraphics[width=0.6\textwidth]{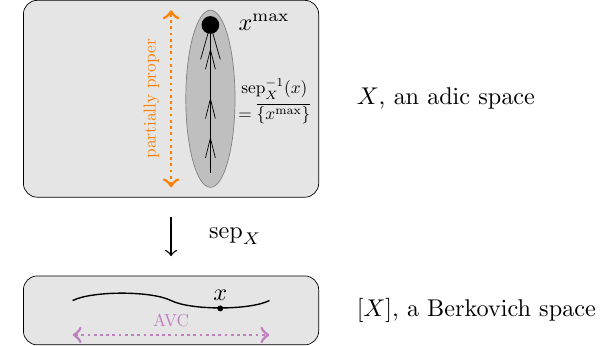}
    \caption{The separation map from the adic space to the Berkovich space [X] and an arc in [X].  Here, AVC allows us to connect different fibers, while partial properness lets us relate points in the fibers.}
    \label{fig:val-criterion}
\end{figure}

Finally we briefly mention some applications, carried out in our companion paper \cite{ALY1P2}, of the theory of geometric coverings to several objects previously highlighted by other authors. For $\tau\in\{\mathrm{adm},\text{\'et},\text{oc}\}$, we define $\Cov^\tau_X$ to consist of \'etale maps $Y\to X$ for which there exists a $\tau$-cover $U\to X$ such that $Y_U\to U$ is the disjoint union of finite \'etale coverings of $U$. Here $\adm,\et,$ or $\oc$ denotes the usual (i.e.\ `admissible'), \'etale, or overconvergent (i.e.\@ partially proper open) Grothendieck topology on $X$, respectively. Thus $\Cov^\text{oc}_X$ is the category of de Jong covering spaces, and each $\Cov^\tau_X$ is a~subcategory of $\Cov_X$.
\begin{enumerate}
    \item We show that if $X$ is connected, then the category of disjoint unions of objects of  $\Cov^\tau_X$ is a tame infinite Galois category. For $\tau=\text{oc}$, this is contained in  \cite{deJongFundamental}, and for $\tau=\text{adm}$ it answers a question asked by de Jong in the same paper. 
    \item De Jong also asked whether $\Cov^\text{adm}_X = \Cov^\text{oc}_X$. We prove that this is not true for an annulus in equal characteristic $p>0$, using an explicit construction with Artin--Schreier coverings. Recently Gaulhiac in \cite{Gaulhiac} has adapted our example to the mixed characteristic setting.
    \item In addition, we connect the larger category $\Cov^\et_X$, to the pro-\'etale topology introduced by Scholze in \cite{Scholzepadic} by obtaining an equivalence of categories 
    \[
        \Cov^\et_X \simeq \cat{Loc}(X_\proet)
    \]
    with the category of locally constant sheaves on $X_\proet$. 
\end{enumerate}

\subsection*{Acknowledgements}

We would like to thank Jeremy Brazas, Antoine Ducros, Gabriela Guzman, David Hansen, Johan de~Jong, Kiran Kedlaya, Nakyung `Rachel' Lee, Emmanuel Lepage, Shizhang Li, Wiesława Nizioł, Jérôme Poineau, Peter Scholze, and Maciej Zdanowicz for useful conversations and comments on the preliminary draft of this manuscript. We especially thank Ofer Gabber for pointing out several inaccuracies in previous versions of this manuscript, and suggesting fixes.

This work is a part of the project KAPIBARA supported by the funding from the European Research Council (ERC) under the European Union’s Horizon 2020 research and innovation programme (grant agreement No 802787). In addition, the final named author was partially supported by a JSPS fellowship during the final editing of this document.

\section{Preliminaries on valuative spaces and separated quotients}
\label{s:topol}

In this section we collect material of a purely topological nature, mostly based on \cite[Chapter~0]{FujiwaraKato}. We encourage the reader to skip this section upon first reading.

\subsection{Valuative spaces and universal separated quotients}
\label{ss:val-sep-quot}

The notion of a valuative space captures the key topological properties of adic spaces.

\begin{definition}[{\cite[Chapter 0, Definition 2.3.1]{FujiwaraKato}}]
    A topological space $X$ is called \emph{valuative} if it is locally spectral and if for each $x$ in $X$ the set of generizations of $x$ is totally ordered by the generization relation.
\end{definition}

Every point $x$ of a valuative space $X$ has a unique maximal generization, denoted $x^\mx$ (see \cite[Chapter 0, Remark 2.3.2 (1)]{FujiwaraKato}). We call a point $x$ of $X$ \emph{maximal} if $x=x^\mx$.

\begin{definition}[{\cite[Chapter 0, \S2.3.(c)]{FujiwaraKato}}] 
    Let $X$ be a valuative space. Then, the \emph{universal separated quotient} of $X$, denoted $[X]$, is the quotient space $X/\sim$ where $x\sim y$ if $x^\mx=y^\mx$. The quotient map $\sep_X\colon X\to [X]$ is called the \emph{separation map}. 
\end{definition}

For a point $x$ of $[X]$ we will write $x^\mx$ for the unique maximal point of $X$ in $\sep_X^{-1}(x)$. Note that for a point $x$ in a valuative space $X$ the equality $\sep_X^{-1}(\sep_X(x))=\overline{\{x^\mx\}}$ holds. A map $f\colon Y\to X$ of valuative spaces is \emph{valuative} (see \cite[Chapter 0, Definition 2.3.21]{FujiwaraKato}) if it has the property that $f(y)$ is maximal whenever $y$ is maximal.

By \cite[Chapter 0, Proposition 2.3.9]{FujiwaraKato}, the map $\sep_X:X\to [X]$ is initial amongst continuous maps from $X$ to $T_1$-topological spaces and, in particular, $[X]$ is $T_1$. This implies that the universal separated quotient of $X$ is functorial with respect to continuous maps. It also implies, by thinking about maps to $\{0,1\}$, that the map $\pi_0(X)\to\pi_0([X])$ is a bijection. For a map $f\colon Y\to X$ of valuative spaces, we denote the induced map $[Y]\to [X]$ by $[f]$. One of the themes of this paper it the interplay of the geometry of $X$ and the topology of $[X]$. 

\subsection{Taut valuative spaces}

In this section we discuss a technical condition on a valuative space $X$, namely when such a space is `taut'. Tautness will play an important role in our article since it will imply that $[X]$ and $\sep_X$ have reasonable topological properties.

\begin{definition}[{\cite[Definition 5.1.2 i)]{Huberbook}}] 
    A valuative space $X$ is \emph{taut}\footnote{In the terminology of Fujiwara--Kato, taut means quasi-separated and locally strongly compact (see \cite[Chapter 0, Remark 2.5.6]{FujiwaraKato}).} if it is quasi-separated and if the closure of every quasi-compact open subset of $X$ is quasi-compact.
\end{definition}

This definition is relatively tame since \cite[Chapter 0, Proposition 2.5.15]{FujiwaraKato} shows that if $X$ is a valuative space which is quasi-separated and paracompact (in the sense of \cite[Chapter~0, Definition 2.5.13]{FujiwaraKato}) then $X$ is automatically taut. 

\begin{prop}[{\cite[Chapter 0, Theorem 2.5.7 and Corollary 2.5.9]{FujiwaraKato}}] \label{taut properties} 
    Let $X$ be a~taut valuative space. Then, the following statements are true:
    \begin{enumerate}[(a)]
        \item \label{taut properties 1} the universal separated quotient $[X]$ is locally compact and Hausdorff,
        \item \label{tau properties 2} the map $\sep_X:X\to [X]$ is universally closed.
    \end{enumerate}
\end{prop}

Tautness is also useful for resolving a possible subtlety concerning the separation map and open subspaces of a valuative space. 

\begin{prop} \label{prop:taut-open}
    Let $X$ be a taut valuative space and $U$ a retrocompact open subspace. Then, the natural map $[U]\to [X]$ is a homeomorphism onto $\sep_X(U)$.
\end{prop}

\begin{proof} 
The inclusion map $j\colon U\to X$ is taut (see Definition \ref{taut map def}) since it is quasi-compact and quasi-separated. Thus, the natural map $[j]\colon [U]\to [X]$ is a map of Hausdorff spaces. It is clearly continuous and injective and has image $\sep_X(U)$. Thus, it suffices to prove that $[j]$ is closed, but this follows from \cite[Chapter 0, Proposition 2.3.27]{FujiwaraKato}.
\end{proof}

In particular, if $U$ is a retrocompact open subspace of a taut valuative space $X$, then we shall often identify  $\sep_X(U)$ and $[U]$ without comment. 

Finally, we record a result which shows that, in good situations, the universal separated quotients commutes with pullbacks of maps along open embeddings.

\begin{definition}[{\cite[Definition 5.1.2 ii)]{Huberbook}}] \label{taut map def} 
    Let $f\colon Y\to X$ be a morphism of valuative spaces. Then, $f$ is called \emph{taut} if it is locally quasi-compact (see \cite[Chapter 0, Definition~2.2.24]{FujiwaraKato}) and for every taut open subspace $U$ of $X$ the space $f^{-1}(U)$ is taut.
\end{definition}
 
\begin{prop} \label{commuting sep and inverse image} 
    Let $f\colon Y\to X$ be a valuative and taut map between valuative spaces, and let $U\subseteq X$ be retrocompact open. Then the natural map
    \[ 
        [Y\times_X U]\to [Y]\times_{[X]}[U]
    \]
    is a homeomorphism.
\end{prop}

\begin{proof}
We first claim that $\sep_Y(f^{-1}(U))=[f]^{-1}(\sep_X(U))$. To see this, we identify $[X]$ with maximal points of $X$, and therefore $\sep_X$ with the map $x\mapsto x^\mx$, and similarly for $Y$. Since $f$ is valuative, we have $f(y)^\mx = f(y^\mx)$, and $[f]$ corresponds to the restriction of $f$ to maximal points. The left hand side consists of $y=z^\mx$ where $f(z)\in U$, but (since $f$ is valuative and continuous and $U$ is open) this is equivalent to $f(y)\in U$.

By Proposition~\ref{prop:taut-open}, we have $[U] = \sep_X(U)$, and hence the target of the above map is $[f]^{-1}(\sep_X(U))$. By the previous paragraph, this equals $\sep_Y(f^{-1}(U))$. But $f^{-1}(U)$ is retrocompact in $Y$ (see \cite[Chapter 0, Proposition 2.2.25]{FujiwaraKato}), and again by Proposition~\ref{prop:taut-open} we have $\sep_Y(f^{-1}(U)) = [f^{-1}(U)] = Y\times_X U$.
\end{proof}

\subsection{Overconvergent open subsets, overconvergent interiors}
\label{ss:oc}

Another aspect of valuative spaces that will be important to us is the refinement of the notion of open subset of a~valuative space $X$ that the separation map $\sep_X:X\to [X]$ allows. 

\begin{prop}[{\cite[Chapter 0, Proposition 2.3.13]{FujiwaraKato}}] \label{oc equiv} 
    Let $X$ be a valuative space and let $U$ be an open subset of $X$. Then, the following conditions are equivalent:
    \begin{enumerate}[(a)]
        \item There exists an open subset $V\subseteq [X]$ such that $U=\sep_X^{-1}(V)$.
        \item The equality $\sep_X^{-1}(\sep_X(U))=U$ holds.
        \item For all $x$ in $U$ one has that $\overline{\{x^\mx\}}\subseteq U$.
    \end{enumerate}
\end{prop}

\begin{definition} 
    Let $X$ be a valuative space. An open subset $U$ of $X$ is called \emph{overconvergent} if any of the equivalent conditions of Proposition~\ref{oc equiv} holds.
\end{definition}

Overconvergent open subsets of $X$ form a topology called the \emph{oc topology}. For a subset $F\subseteq X$ we denote by $\interior_X(F)$ the interior in this topology, and call it the \emph{overconvergent interior} of $F$.

\begin{prop} \label{interior equivalent} 
    Let $X$ be a taut valuative space, $U$ an open subset of $X$, and $S$ a~subset of $U$. Consider the following conditions:
    \begin{enumerate}[(1)]
        \item For all $x$ in $S$ one has that $\ov{\{x^\mx\}}\subseteq U$,
        \item $\sep_X^{-1}(\sep_X(S))$ is contained in $U$,
        \item $S$ is contained in $\interior_X(U)$,
        \item the topological interior of $\sep_X(U)$ in $[X]$ contains $\sep_X(S)$.
    \end{enumerate}
    Then, $(1)\Longleftrightarrow(2)\Longleftrightarrow(3)\Longrightarrow (4)$
\end{prop}

\begin{proof} 
To see that (1), (2), and (3) are equivalent we may assume that $S=\{x\}$. To prove the equivalence of (1) and (2) we merely apply the equality $\ov{\{x^\mx\}}=\sep_X^{-1}(\sep_X(x))$. We note that the equivalence of (1) and (3) is given in \cite[Chapter 0, Proposition 2.3.29]{FujiwaraKato} when $U$ is assumed quasi-compact (note that condition $(\ast)$ in loc.\@ cit.\@ is automatic since $X$ is assumed taut). But, by Proposition \ref{taut properties} and the equality $\ov{\{x^\mx\}}=\sep_X^{-1}(\sep_X(x))$ we see that the subspace $\ov{\{x^\mathrm{\mx}\}}$ is quasi-compact. It is then clear one can remove the quasi-compactness assumption in \cite[Chapter 0, Proposition 2.3.29]{FujiwaraKato}.

That (3) implies (4) is also simple since then $\sep_X(S)\subseteq \sep_X(\interior_X(U))$. As $\interior_X(U)$ is overconvergent we see that $\sep_X(\interior_X(U))$ is open in $[X]$ and contained in $\sep_X(U)$.
\end{proof}

\begin{rem} 
For a counterexample of the equivalence of (3) and (4) take $X$ to be the closed unit disk over $K$, and $U=X-\{x\}$ where $x$ is a type $5$ point in the closure of the Gauss point $\eta$ of $X$. Then, $\sep_X(U)=[X]$ and so $\sep_X(\eta)$ is contained in the interior of $\sep_X(U)$. But, since $x\in\ov{\{\eta\}}$ we see from (1) that $\eta$ is not in $\interior_X(U)$.
\end{rem}

Let $X$ be a taut valuative space, $U$ an open subset, and $S$ a subset of $U$. Below we shall often use the terminology that $U$ is an \emph{open oc neighborhood of $S$} to mean that $S$ is contained in $\interior_X(U)$. By the above this is equivalent to the claim that $U$ contains $\sep_X^{-1}(\sep_X(S))$. So we harmlessly abuse notation and say that for an open subset $U\subseteq X$ and a subset $S$ of $[X]$ that $U$ is an \emph{open oc neighborhood of $S$} if $U$ contains $\sep_X^{-1}(S)$.

\begin{prop}[{cf.\@ \cite[Lemma 8.1.5]{Huberbook}}] \label{taut qc oc neighborhood basis} 
    Let $X$ be a taut valuative space and $\Sigma$ a quasi-compact subset of $X$. Then, the set of quasi-compact open oc neighborhoods of $\Sigma$ form an open oc neighborhood basis of~$\Sigma$.
\end{prop}

\begin{proof} 
Assume first that $\Sigma=\{x\}$ is a point. Let $U$ be an open oc neighborhood of $x$. Note then that necessarily $U$ contains $x^\mx$. We see then by Proposition~\ref{interior equivalent} that $\ov{\{x^\mx\}}\subseteq U$. Since $\ov{\{x^\mx\}}$ is quasi-compact (as observed in the proof of Proposition~\ref{interior equivalent}), there exists a quasi-compact open subset $V$ of $U$ such that $\ov{\{x^\mx\}}\subseteq V$. So then, by Proposition~\ref{interior equivalent} we're done. Now, for arbitrary $\Sigma$, we can apply the above to all $x\in \Sigma$ and employ quasi-compactness of $\Sigma$ to find a subcover by finitely many such $\interior_X(V)$.
\end{proof}

\section{Preliminaries on rigid geometry}
\label{s:rigid-prelim}

In this section we review some rigid geometry, mostly based on \cite{Huberbook} and \cite{FujiwaraKato}.

\subsection{\texorpdfstring{Adic spaces and rigid $K$-spaces}{Adic spaces and rigid K-spaces}}
\label{ss:rigid-k-sp}

In this subsection we set our terminology concerning adic spaces and rigid spaces. We only consider adic spaces $X$ for which every affinoid open subspace is of the form $\Spa(A,A^+)$ where $(A,A^+)$ is a complete, analytic, and strongly Noetherian (see \cite[\S 1.1]{Huberbook}) Huber pair. For a Huber pair of the form $(A,A^\circ)$ we abbreviate $\Spa(A,A^\circ)$ to $\Spa(A)$.

Let $K$ be a non-archimedean field, i.e.\ a field complete with respect to a rank one valuation. By a \emph{rigid $K$-space} we mean an adic space $X$ locally of finite type over $\Spa(K)$. A \emph{rigid $K$-curve} is a rigid $K$-space which is of pure dimension $1$.

By an \emph{affinoid field} we mean a Huber pair $(L,L^+)$ where $L$ is a field, $L^+$ is a valuation ring, and $L$ has the valuation topology. For each point $x$ of an adic space $X$ one obtains an affinoid field $(k(x),k(x)^+)$.

We note here that the underlying topological space of any adic space is valuative, and  every morphism $f\colon Y\to X$ of adic spaces is valuative and locally quasi-compact (this follows from \cite[Theorem 3.5 (i)]{Huber93}, \cite[(1.1.9)]{Huberbook}, and \cite[Lemma 1.1.10 (4)]{Huberbook}). Moreover, by \cite[Theorem C.6.12]{FujiwaraKato} if $X$ is taut then $[X]$ is precisely the underlying topological space of the Berkovich space $X^\mathrm{Berk}$ associated to $X$ (see \cite[\S8.3]{Huberbook}]).

For an adic space $X$, we denote by $\Et_X$ (resp.\@ $\FEt_X$) the full subcategory of the category of adic spaces over $X$ consisting of \'etale (resp.\@ finite \'etale) morphisms.

\subsection{Geometric points, fiber functors, and algebraic fundamental groups}
\label{ss:geometric-points-and-fundamental-groups}

We now recall geometric points and their associated fiber functors.

\medskip
\paragraph{Geometric points} 

By a \emph{geometric point} of an adic space $X$ we mean a morphism of adic spaces $\ov{x}\colon \Spa(L,L^+)\to X$ where $(L,L^+)$ is an affinoid field and $L$ is separably closed (or, equivalently, algebraically closed). Its \emph{anchor point} is the image of the unique closed point of $\Spa(L,L^+)$. For a geometric point we write $\ov{x}^{\rm max}$ for the associated \emph{maximal geometric point} obtained as the composition $\Spa(L)\to \Spa(L,L^+)\to X$. 

Suppose that $\ov{x}\colon\Spa(L,L^+)\to X$ is a geometric point with anchor point $x$. Let $F_0$ denote the separable closure of $k(x)$ in $L$, which is separably closed since $L$ is. There is a valuation ring $F_0^+$ of $F_0$, unique up to Galois conjugacy, such that $F_0^+\cap k(x)=k(x)^+$. The completion $(F,F^+)$ of $(F_0,F_0^+)$ in $(L,L^+)$ is a Huber pair and we call the choice of such a pair a \emph{separable closure} of $x$ in $\ov{x}$. By an \emph{equivalence} $\ov{x}_1\to \ov{x}_2$, which only exists if the anchor points of these geometric points are the same point $x$ of $X$, we mean an isomorphism $(F_1,F_1^+)\to (F_2,F_2^+)$ of separable closures $(F_i,F_i^+)$ of $x$ in $\ov{x}_i$. 

\medskip

\paragraph{Fiber functors and algebraic fundamental groups}

Let $\ov{x}\colon\Spa(L,L^+)\to X$ be a~geometric point of an adic space $X$. 

\begin{definition}
    The \emph{fiber functor} associated to $\ov{x}$ is the functor 
    \begin{equation*}
        F_{\ov{x}}\colon\Et_X\to \cat{Set},\qquad (Y\to X)\mapsto \Hom_X(\Spa(L,L^+),Y) = \pi_0(Y_{\ov x})
    \end{equation*}
    We shall also use the notation $F_{\ov{x}}$ to refer to the restriction to any subcategory of $\Et_X$.
\end{definition}

Note that $F_{\ov{x}}(Y)$ is precisely the set of all liftings of $\ov{x}$ along $Y\to X$. There is a natural functorial association $F_{\ov{x}}(Y)\to Y_{\ov{x}}$ given by associating a lifting $\ov{y}$ to the anchor point of the geometric point $\Spa(L,L^+)\to Y_{\ov{x}}$ induced by $\ov{y}$. This map is not a bijection, unless $\ov{x}$ is a maximal geometric point, but does induce a functorial bijection $F_{\ov{x}}(Y)\to \pi_0(Y_{\ov{x}})$.

\begin{prop}
    Suppose $X$ is connected and let $\ov{x}$ and $\ov{y}$ be geometric points of $X$. Then $(\FEt_X,F_{\ov{x}})$ is a Galois category and the fiber functors $F_{\ov{x}}$ and $F_{\ov{y}}$ are isomorphic. 
\end{prop}

Using this we can define the algebraic fundamental group.

\begin{definition}
    Let $X$ be a connected adic space and let $\ov{x}$ and $\ov{y}$ be geometric points of $X$. Then, the set of \emph{\'etale paths} from $\ov{x}$ to $\ov{y}$ is the set
    \begin{equation*}
        \pi_1^\alg(X;\ov{x},\ov{y})=\mathrm{Isom}_{\FEt_X}(F_{\ov{x}},F_{\ov{y}})
    \end{equation*}
    When $\ov{x}=\ov{y}$ we shorten the notation to $\pi_1^\alg(X,\ov{x})$ and call the resulting profinite group the \emph{algebraic fundamental group} of $X$.
\end{definition}

We use the following well-known result freely (see \cite[Example 1.6.6 ii)]{Huberbook}): for an affinoid adic space $X=\Spa(R, R^+)$, the category $\FEt_X$ is equivalent to the category $\FEt_{\Spec(R)}$. In particular, for $X$ connected we have $\pi_1^\alg(X, \ov x,\ov y)\simeq \pi_1^\et(\Spec(R), \ov x, \ov y)$.

\subsection{Partially proper morphisms}
\label{ss:pp}

In this section we recall the some important properties related to partially proper morphisms (see \cite[Definition 1.3.3]{Huberbook})\footnote{One should note that partially proper morphisms appear in the theory of Berkovich spaces under the name `boundaryless'.}.

First, it's thematically important to note that partial properness (and separatedness) can be understood in terms of a certain valuative criterion. 

\begin{prop}[{\cite[Lemma 1.3.10]{Huberbook}, \cite[Proposition 6.13]{LudwigNotes}}] \label{valuative criterion prop}
    Let $f\colon Y\to X$ be a~morphism of adic spaces which is quasi-separated  and locally of $^+$weakly finite type (resp.\@ locally of weakly finite type). Then, the following properties are equivalent.
    \begin{enumerate}[(1)]
        \item The morphism $f$ is partially proper (resp.\@ separated),
        \item \label{valuative criterion prop:2} for every affinoid field $(L,L^+)$ and every morphism $\Spa(L,L^+)\to X$ the morphism $\Hom_X(\Spa(L,L^+),Y)\to\Hom_X(\Spa(L),Y)$ is a bijection (resp.\@ injection),
        \item \label{valuative criterion prop:3} for every Huber pair $(R,R^+)$ and every morphism $\Spa(R,R^+)\to X$ the morphism $\Hom_X(\Spa(R,R^+),Y)\to \Hom_X(\Spa(R),Y)$ is a bijection (resp.\@ injection).
    \end{enumerate}
\end{prop}

We then record two useful facts about the relationship between partially proper morphisms and the topology of valuative spaces.

\begin{prop}[{\cite[Lemma 5.1.4.ii)]{Huberbook}}] \label{pp implies taut} 
    Let $f\colon Y\to X$ be a partially proper morphism of adic spaces. Then, $f$ is taut.
\end{prop}

\begin{prop} \label{pp open embedding} 
    Let $X$ be an adic space and $U$ an open subset of $X$. Then, the inclusion morphism $U\hookrightarrow X$ is partially proper if and only if $U$ is an overconvergent open.
\end{prop} 

\subsection{Huber's universal compactifications and tautness of curves}
\label{taut subsection}

Since a large portion of our theory relies on rigid $K$-curves, it is useful to know that they are taut. The proof of this fact employs Huber's universal compactifications.

\begin{definition}[{\cite[Definition 5.1.1]{Huberbook}}]
     A \emph{universal compactification} of a map of adic spaces $U\to X$ is a locally closed embedding $U\to U^{\rm univ}_{/X}$ over $X$ with $U^{\rm univ}_{/X}\to X$ partially proper, and such that for every factorization $U\to U'\to X$ with $U'\to X$ partially proper there exists a unique morphism $U^{\rm univ}_{/X}\to U'$ under $U$ and over $X$. 
\end{definition}

\begin{thm}[Huber, {\cite[Theorem~5.1.5]{Huberbook}}] \label{thm:Huber-univ-compact}
    Let $U\to X$ be a taut and separated map of rigid $K$-spaces. Then, a universal compactification $j:U\to U^{\rm univ}_{/X}$ exists, is a~quasi-compact open embedding, and every point of $U^{\rm univ}_{/X}$ is a specialization of a point of~$U$.
\end{thm}

If $U$ is a rigid $K$-space, taut and separated over $\Spa(K)$, then we shall shorten the notation $U^{\rm univ}_{/\Spa(K)}$ to $U^{\rm univ}$. An example of such an object is given in the following.

\begin{example} \label{ex:Piotr-comp-disc}
    For $\mathbf{D}^1_K = \Spa(K\langle X\rangle)$, we have $(\mathbf{D}^1_K)^{\rm univ} = \mathbf{D}^1_K\cup \{\nu_{0, 1^+}\}$ where $\nu_{0,1^+}$ is a~certain rank two valuation in $\mathbf{A}^{1,\mathrm{an}}_K$ in the closure of the Gauss point 
\end{example}

We would like to explicitly link Huber's definition of universal compactifications with the approach taken by Scholze in \cite[\S 18]{ScholzeDiamonds} in the context of $v$-sheaves. In the lemma below, we denote by $\mc{A}$ the category of affinoid adic spaces $T=\Spa(R, R^+)$ such that $R$ is topologically of finite type over some non-archimedean field.

\begin{lem} \label{lem:univcomp-reflex}
    Let $U\to X$ be a separated and taut morphism of rigid $K$-spaces. Then, we have an isomorphism of presheaves on $\mc{A}$:
    \begin{equation*}
        \Hom(T, U^{\rm univ}_{/X}) \simeq \Bigg \{(f,\tilde f)\,:\, \vcenter{\xymatrix{T^\circ\ar[r]^-{\tilde f} \ar[d] & U\ar[d] \\ T\ar[r]_-f & X}}\emph{ commutes}\Bigg\}
    \end{equation*}
    where $T=\Spa(R, R^+)$ and $T^\circ=\Spa(R)$.
\end{lem}

\begin{proof}
It suffices to show that for $T=\Spa(R,R^+)\to X$ an object of $\mc{A}_{/X}$ there is a~functorial identification $\Hom_X(T,U^{\rm univ}_{/X})\simeq \Hom_X(T^\circ,U)$. Let us note that one has a~natural injective map of sets $\Hom_X(T^\circ,U)\to \Hom_X(T^\circ,U^{\rm univ}_{/X})$. We also have a bijection
\[
    \Hom_X(T,U^{\rm univ}_{/X})\to \Hom_X(T^\circ,U^{\rm univ}_{/X})
\]
by Proposition \ref{valuative criterion prop}~\ref{valuative criterion prop:3} applied to the partially proper map $U^{\rm univ}_{/X}\to X$. In particular, we get an injective map
\begin{equation*}
    \Hom_X(T^\circ,U)\to \Hom_X(T^\circ,U^{\rm univ}_{/X})\simeq \Hom_X(T,U^{\rm univ}_{/X})
\end{equation*}
This map is clearly functorial, and thus it suffices to show that it is a bijection. To do this, it suffices to show that for every map $f\colon T\to U^{\rm univ}_{/X}$ over $X$, that $f(T^\circ)\subseteq U$. 

To see this, note that by Theorem \ref{thm:Huber-univ-compact} that every maximal point of $U^{\rm univ}_{/X}$ is a point $U$. In particular, since $f$ is valuative, we see that $f^{-1}(U)$ contains every maximal point of $T$. But, by Theorem \ref{thm:Huber-univ-compact} the map $U\to U^{\rm univ}_{/X}$ is quasi-compact. Thus, $f^{-1}(U)\to T$ is quasi-compact, and thus $V=f^{-1}(U)$ is quasi-compact. 

It remains to show $T^\circ\subseteq V$. Since $T$ is quasi-separated, $V\cap T^\circ$ is a quasi-compact open subset of $T^\circ$. By the above, we see that $V\cap T^\circ$ contains all maximal and hence all classical points of $T^\circ$, and hence is equal to it by \cite[Theorem~4.3]{Huber93}.
\end{proof}

Note that if $X$ is a rigid $K$-space, then for any morphism of rigid $K$-spaces $U\to V$ which is separated and taut over $X$ Lemma \ref{lem:univcomp-reflex} implies that one obtains an induced morphism of adic spaces $f^{\rm univ}_{/X}\colon U^{\rm univ}_{/X}\to V^{\rm univ}_{/X}$. If $X=\Spa(K)$ we shall shorten this notation to just $f^{\rm univ}:U^{\rm univ}\to V^{\rm univ}$.

\begin{prop} \label{prop:-univcomp}
Let $U$ and $X$ be a rigid $K$-spaces. Let $U\to X$ be a separated and taut morphism. Then, the following statements hold true.
\begin{enumerate}[(a)]
    \item \label{prop:-univcomp 1} Suppose that $U\to X$ is an open immersion. Then, $U^{\rm univ}_{/X}\to X$ is a homeomorphism onto the closure of $U$ in $X$.
    \item \label{prop:-univcomp 2} If $U$ is taut and separated over $\Spa(K)$, then we have an injective map $U^{\rm univ}_{/X}\to U^{\rm univ}$ under $U$.
    \item For $U=\Spa(A)$ we have a, functorial in $A$, identification $U^{\rm univ}=\Spa(A, A')$ where $A'\subseteq A^\circ$ is the integral closure of the $\h_K$-subalgebra of $A$ generated by $A^{\circ\circ}$.
    \item \label{prop:-univcomp 4} If $U\to V$ is a finite morphism of affinoid $K$-spaces, then the induced map $U^{\rm univ}\to V^{\rm univ}$ is finite.
\end{enumerate}
\end{prop}

\begin{proof} 
To prove the first statement, note that by Lemma~\ref{lem:univcomp-reflex} the map 
\[
    \Hom(\Spa(R,R^+),U^{\rm univ}_{/X})\to \Hom(\Spa(R,R^+),X)
\]
is injective for all objects $\Spa(R,R^+)$ of $\mc{A}$. Thus, $U^{\rm univ}_{/X}\to X$ is injective. To see the image is $\ov{U}$, note that by Lemma~\ref{lem:univcomp-reflex} applied to $T=\Spa(k(x), k(x)^+)$ for $x\in X$ shows that $x^\mx\in U$ if and only if $T\to X$ lifts to $U^{\rm univ}_{/X}$, in which case that lifting is unique. Combining this with \cite[Chapter 0, Corollary 2.2.27]{FujiwaraKato} shows that the image is indeed $\ov{U}$. The map $U^{\rm univ}_{/X}\to X$ is quasi-compact by \cite[Corollary 5.1.6]{Huberbook}, and thus it is a~homeomorphism onto $\ov{U}$ by \cite[Lemma 1.3.15]{Huberbook}

For the second statement, using the notation of Lemma~\ref{lem:univcomp-reflex}, the map $U^{\rm univ}_{/X}\to U^{\rm univ}$ corresponds via the Yoneda lemma to the map $(f, \tilde f)\mapsto (\pi_X\circ f, \tilde f)$ (where $\pi_X\colon X\to \Spa(K)$ is the structure map). This map is clearly under $U$, and is injective by Proposition \ref{valuative criterion prop}~\ref{valuative criterion prop:3} since $U\to X$ is separated

For the third statement, note that every map $(A, A^\circ)\to (R, R^\circ)$ over $(K, \cO_K)$ sends $A'$ to $R'$, and $R'\subseteq R^+$ because $R'$ is the smallest subring of integral elements of $R$ containing $\cO_K$. This shows that $(A,A')$ represents the correct presheaf as in Lemma \ref{lem:univcomp-reflex}.

To see the final claim, write $V=\Spa(B)$ and $U=\Spa(A)$. We know that $B\to A$ is finite and hence that $B^\circ\to A^\circ$ is integral by \cite[\S6.3.4, Proposition 1]{BGR}. We claim that this implies $B'\to A'$ is integral, which is sufficient to prove our claim. But, for this it suffices to show that every element of $A^{\circ\circ}$ satisfies a monic equation with non-leading coefficients in $B^{\circ\circ}$. Let $f\in A^{\circ\circ}$, then we can write $f^m=cg$ where $m\geq 1$ and $c\in K$ with $|c|<1$ and $g\in A^\circ$ (e.g.\@ apply the contents of \cite[\S6.2.3]{BGR}). Let $P=T^n + a_1 t^{n-1} + \ldots + a_n \in B^\circ[T]$ be a~monic polynomial with $P(g)=0$. Then $f$ satisfies the monic polynomial $c^nP(T^m/c)$ whose non-leading coefficients lie in $B^{\circ\circ}$.
\end{proof}

\begin{cor} \label{cor:curvecomp}
    Let $U$ be a quasi-compact and quasi-separated rigid $K$-curve. Then $U^{\rm univ}\setminus U$ is finite. 
\end{cor}

\begin{proof}
By \cite[Chapter 0, Corollary 2.2.27]{FujiwaraKato} and Theorem \ref{thm:Huber-univ-compact} we see $\ov{U}=U^{\rm univ}$. Thus, we must show that $\ov{U}\setminus U$ is finite. If $\{V_i\}$ is a finite affinoid open cover of $U$, it suffices to show that $\ov{V_i}\setminus V_i$ is finite or all $i$. By combining Proposition \ref{prop:-univcomp} \ref{prop:-univcomp 1} and Proposition \ref{prop:-univcomp 2} it suffices to show that $V_i^{\rm univ}\setminus V_i$ is finite. Thus, we may assume that $U$ is affinoid.

Let $f\colon U\to D:=\mathbf{D}^1_K$ be a finite surjective map (see \cite[\S6.1.2, Corollary 2]{BGR}). By Proposition~\ref{prop:-univcomp} \ref{prop:-univcomp 4}, the induced map $f^{\rm univ}\colon U^{\rm univ}\to D^{\rm univ}$ is finite. As $f$ is finite, $U\to (f^{\rm univ})^{-1}(D)$ is finite, and thus has closed image, and therefore $\ov{U}\cap (f^{\rm univ})^{-1}(D)=U$. But, by combining \cite[Chapter 0, Corollary 2.2.27]{FujiwaraKato} and Theorem \ref{thm:Huber-univ-compact} we know that $\ov{U}=U^{\rm univ}$. Thus, $(f^{\rm univ})^{-1}(D)=U$, and so 
\[
    U^{\rm univ}\setminus U = (f^{\rm univ})^{-1}(D^{\rm univ}\setminus D) = (f^{\rm univ})^{-1}(\nu_{0, 1^+}) \quad \text{(see Example~\ref{ex:Piotr-comp-disc})}
\]
which is a finite set because $f^{\rm univ}$ has finite fibers by \cite[Lemma 1.5.2]{Huberbook}.
\end{proof}

We now prove that every quasi-separated rigid $K$-curve is taut.

\begin{prop} \label{prop:curves-are-taut}
    Let $X$ be a quasi-separated rigid $K$-curve. Then, $X$ is taut.
\end{prop}

\begin{proof}
As $X$ is quasi-separated, it suffices to show that for every quasi-compact open $U\subseteq X$, the closure $\ov{U}$ is quasi-compact. By Proposition \ref{prop:-univcomp} \ref{prop:-univcomp 1}, $\ov U \setminus U = U^{\rm univ}_{/X}\setminus U$, which by Proposition~\ref{prop:-univcomp} \ref{prop:-univcomp 2} is contained in $U^{\rm univ}\setminus U$. The latter set is finite by Corollary~\ref{cor:curvecomp}. So $\ov{U} = U\cup (\ov{U}\setminus U)$ is quasi-compact, as the union of a quasi-compact open and a finite set.
\end{proof}

\subsection{Goodness of curves}
\label{ss:oc-taut-good}

While Proposition \ref{taut qc oc neighborhood basis} implies that every maximal point of a taut adic space is contained in a quasi-compact open oc neighborhood, it does not imply that one can take such a quasi-compact open oc neighborhood to be affinoid. 

\begin{definition}
    A taut adic space $X$ is called \emph{good} (cf.\@ \cite[Remark 1.2.16]{BerkovichEtale}) if every point $x$ of $X$ has an open oc neighborhood basis of affinoid open oc neighborhoods.
\end{definition} 

A taut adic space $X$ is good if and only if every point $x$ admits some affinoid oc open neighborhood. It will also be useful for our construction of geometric intervals on rigid $K$-curves (see Proposition \ref{geometric intervals exist}) to recall that any reasonably nice rigid $K$-curve is automatically good.

\begin{prop}[{cf.\@ \cite[Corollary 3.4]{deJongFundamental}}] \label{curves are good} 
    Let $X$ be a smooth and separated rigid $K$-curve and let $\Sigma$ be a connected quasi-compact subset of $X$. Suppose that $\sep_X(\Sigma)$ is not a connected component of $[X]$, then the set of affinoid open oc neighborhoods of $\Sigma$ forms an open oc neighborhood basis of $\Sigma$. In particular, $X$ is good.
\end{prop}

\begin{proof}
Let $U$ be an open oc neighborhood of $\Sigma$. By Proposition \ref{taut qc oc neighborhood basis} we can find a~quasi-compact open oc neighborhood $W$ of $\Sigma$ contained in $U$. By passing to a connected component of $W$ containing $\Sigma$ we may assume without loss of generality that $W$ is connected. By \cite[Th\'{e}or\`{e}me 2]{FresnelMatignon}, $W$ is either affinoid or projective. 

In the former case we are done, so suppose that $W$ is projective, and in particular a connected component of $X$. We claim that there exists a classical point $p$ of $W$ not contained in $\Sigma$. Otherwise $\sep_W(\Sigma)$ contains all classical points of $[W]$, and since these are dense, we have $\sep_W(\Sigma)=[W]$, contradicting that $\sep_X(\Sigma)$ is not a connected component of $[X]$. Replace $W$ again with a quasi-compact open oc neighborhood $W'$ of $\Sigma$ but this time in $W\setminus \{p\}$. Now $W'$ cannot be projective (being contained in the affine $W\setminus\{p\}$), so it is affinoid.
\end{proof}

\subsection{\'Etale and partially proper maps}
\label{ss:et-pp}

In this subsection we discuss some special properties of \'etale and partially proper morphisms of rigid $K$-spaces\footnote{We note as a side remark that as proved in  \cite[p.\ 427]{Huberbook}, \'etale and partially proper morphisms (of taut rigid $K$-spaces) correspond to \'etale morphisms of Berkovich $K$-analytic spaces.}.

\medskip

\paragraph{\'Etale localness of \'etale and partially proper maps}

\begin{prop} \label{prop:etale-etale-local}
    Let $Y\to X$ be a morphism locally of finite type between adic spaces and let $X'\to X$ be an \'etale surjection. If $Y_{X'}\to X'$ is \'etale, then $Y\to X$ is \'etale.
\end{prop}

\begin{proof} 
This follows from Lemma \ref{etale image lemma}.
\end{proof}

\begin{lem} \label{etale image lemma}
    Let $Y\to X$ be a flat surjective map of adic spaces locally of finite type over the adic space $S$. Suppose that $Y\to S$ is \'etale, then $X\to S$ is \'etale.
\end{lem}

\begin{proof}
We first show that $X\to S$ is flat. Let $x$ be a point of $X$ mapping to $s$ in $S$ and let $y$ be a point of $Y$ mapping to $x$. We then have a series of ring maps $\h_{S,s}\to \h_{X,x}\to \h_{Y,y}$. By assumption we know the composition is flat, and that the second map is flat, and thus must be the first by \stacks{02JZ} and \stacks{00HR}. Since $\h_{X,x}\to \h_{Y,y}$ is flat so is $\h_{X,x}/\mf{m}_s\h_{X,x}\to \h_{Y,y}/\mf{m}_s\h_{Y,y}$, where $\mf{m}_s$ is the maximal ideal of $\h_{S,s}$, but since this is a map of local rings it is automatically injective by loc.\@ cit. But, since $Y\to S$ is \'etale we know that $\h_{Y,y}/\mf{m}_s\h_{Y,y}$ is a finite separable field extension of $\h_{S,s}/\mf{m}_s$ and therefore so must be $\h_{X,x}/\mf{m}_s\h_{X,x}$. So, $X\to S$ is \'etale at $x$ by \cite[Proposition 1.7.5]{Huberbook}. Since $x$ was arbitrary the conclusion follows.
\end{proof}

The corresponding statement for partially proper morphisms is slightly more sophisticated, and requires Corollary \ref{boostrap corollary} below.

\begin{prop} \label{prop:pp-etale-local}
    Let $Y\to X$ be a locally of $^+$weakly finite type morphism of adic spaces. Suppose that $X'\to X$ is an \'etale surjection such that $Y_{X'}\to X'$ is partially proper. Then, $Y\to X$ is partially proper.
\end{prop}

\begin{proof}
Write $Y' = Y_{X'}$. Since partially proper morphisms are affinoid open local on the target, we may assume by Corollary \ref{boostrap corollary} that $X'\to X$ is a finite \'etale Galois cover. Once we know $Y\to X$ is separated, the claim then follows from Lemma \ref{image of partially proper} since $Y\to X$ is separated, $Y'\to Y$ is surjective, and the composition $Y'\to X'\to X$ partially proper.

To see that $Y\to X$ is separated, note that as $Y'\to Y$ is surjective, the image of $\Delta_{Y/X}\colon Y\to Y\times_X Y$ is the image of the closed subset $\im(\Delta_{Y'/X'})$ under the finite \'etale map $Y'\times_{X'} Y'\to Y\times_X Y$. Since this map is closed, $\im(\Delta_{Y/X})$ is closed.
\end{proof}

\begin{lem} \label{image of partially proper}
    Let $Y\to X$ be a surjective map of adic spaces over the adic space $S$. Suppose that $Y\to S$ is partially proper and that $X\to S$ is separated and locally of $^+$weakly finite type. Then, $X\to S$ is partially proper.
\end{lem}

\begin{proof} 
We only need to show that $X\to S$ is universally specializing. Let $S'\to S$ be any map of adic spaces. Then, $Y_{S'}\to X_{S'}$ is a surjection. Let $x$ be a point of $X_{S'}$ mapping to the point $s$ of $S'$, and let $y$ be a point of $Y_{S'}$ mapping to $x$. Let $s'$ be a point of $S'$ specializing $s$. Since $Y_{S'}\to S'$ is specializing there exists some $y'$ in $Y$ specializing $y$ which maps to $s'$ under $Y_{S'}\to S'$. Since $Y_{S'}\to X_{S'}$ is continuous the image of $y'$ specializes $x$ and maps to $s'$ under $X_{S'}\to S'$ from where the conclusion follows.
\end{proof}

\paragraph{Local finiteness of \'etale and partially proper maps}

The following result, important in what follows, says that an \'etale and partially proper map of adic spaces is `locally finite'. 

\begin{prop}[{cf.\@ \cite[Proposition 1.5.6]{Huberbook}}]\label{pp etale local structure}
    Let $f\colon Y\to X$ be an \'etale partially proper morphism of adic spaces where $X$ is taut. Then, for any point $y$ of $Y$ the set
    \begin{equation*}
        \left\{W\subseteq Y:\begin{aligned}(1)&\quad W\emph{ is an open oc neighborhood of }y\\ (2)&\quad f|_W:W\to f(W)\emph{ is finite \'etale}\end{aligned}\right\}
     \end{equation*}
    forms an open oc neighborhood basis of $y$.
\end{prop}

\paragraph{Properties of $[f]$ for \'etale and partially proper maps} 

We now study the topological implications for $[f]$ assuming that $f$ is \'etale and partially proper.

\begin{prop}\label{[f] for etale and pp}
    Let $f\colon Y\to X$ be an \'etale and partially proper morphism of adic spaces, and suppose that $X$ is taut. Then, the following properties hold true. 
    \begin{enumerate}[(a)] 
        \item The map $[f]$ is open,
        \item the map $[f]$ has discrete fibers, 
        \item the map $[f]$ satisfies the following condition:
        \begin{enumerate}
            \item[$(\ast)$] For $y\in Y$ and open neighborhood $U$ of $y$ there exists an open neighborhood $V$ of $f(y)$ and a clopen subset $W$ of $f^{-1}(V)$ containing $y$ and contained in $U$.
        \end{enumerate}
    \end{enumerate}
\end{prop}

\begin{proof}
The final statement follows from Lemma~\ref{lem:vapid} below, so we focus only on the first two. To show the first statement, note that if $V\subseteq [Y]$ is open then $\sep_Y^{-1}(V)$ is an overconvergent subset of $Y$. So then, $f(\sep^{-1}(V))$ is open since $f$ is \'etale, but it is also overconvergent since it is closed under specialization since $f$ is specializing and $\sep^{-1}(V)$ is closed under specialization. Thus, $\sep_X(f(\sep_Y^{-1}(V)))$ is an open subset of $[X]$. But, since the separation maps are surjective this is equal to $[f](V)$ as desired.

We first address the second statement in the case $f$ is finite \'etale. Since $f$ is valuative it suffices to show that $f$ itself is quasi-finite, but this follows from \cite[Lemma 1.5.2]{Huberbook}. Let $y\in [Y]$ and let $W$ be as in Proposition~\ref{pp etale local structure} (for the point $y^\mx\in Y$). By the finite \'etale case, the assertion holds for $W\to f(W)$. Then, $\sep_X(W)$ is a neighborhood of $y$ which intersects $[f]^{-1}(f(y))$ at finitely many points. So, we're done.
\end{proof}

\begin{lem} \label{lem:vapid}
    Let $f\colon Y\to X$ be a map with discrete fibers between Hausdorff spaces, with $Y$ locally compact. Then, condition $(\ast)$ if Proposition~\ref{[f] for etale and pp} holds.
\end{lem}

\begin{proof} 
If $y$ is a point of $Y$ and $Y_0\subseteq Y$ a closed subset such that $y$ is in the interior of $Y_0$ and $f|_{Y_0}$ satisfies the assertion $(\ast)$ at $y$, then so does $f$. Indeed, let $U$ be an open neighborhood of $y$ in $Y$ and let $U_0$ be the intersection of $U$ and the interior of $Y_0$. This is an open neighborhood of $y$ in $Y_0$, and hence there exists an open neighborhood $V$ of $f(y)$ in $X$ and a clopen subset $W\subseteq f^{-1}(V)\cap Y_0$ containing $y$ and contained in $U_0$ and hence in $U$. Then $W$ is closed in $f^{-1}(V)$. It is also open in $Y$, and hence open in $f^{-1}(V)$.

So then, let $y\in Y$ and let $x=f(y)$. Since $Y$ is locally compact, there exists an open neighborhood $U_1$ of $y$ such that $\ov U_1$ is compact. In particular, $\ov U_1\cap f^{-1}(x)$ is finite (being compact and discrete). Because $Y$ is Hausdorff, we can find an open neighborhood $U_0\subseteq U_1$ of $y$ such that $\ov U_0\cap f^{-1}(x) = \{y\}$. By the first paragraph of this proof applied to $Y_0=\ov U_0$, it is enough to show that $Y_0\to X$ satisfies $(\ast)$ at $y$. In other words, we may assume that $Y$ is compact and $f^{-1}(x)=\{y\}$. Let $U$ be an open neighborhood of $y$, and let $Z=Y\setminus U$. Then $f(Z)$ is closed and does not contain $x$, and hence $V=X\setminus f(Z)$ is an open neighborhood of $x$. Then $f^{-1}(V)\subseteq U$, and we can take $W=f^{-1}(V)$.
\end{proof}

\paragraph{Bounding cardinality of \'etale and partially proper maps} 

The following will be used to prove that the category of geometric coverings of a rigid $K$-space $X$ is not too big.

\begin{lem} \label{lem:card-bound}
    Let $X$ be a rigid $K$-space. Then, there exists a cardinal $\kappa$ such that for every \'etale and partially proper map $Y\to X$ with $Y$ connected, the cardinality of $|Y|$ is less than $\kappa$.
\end{lem}

\begin{proof}
The cardinality of any affinoid rigid $K$-space is bounded by some $\kappa_0$ which depends only on $K$. Indeed, by Noether normalization it suffices to take $\kappa_0$ bigger than the cardinality of $|\mathbf{D}^n_K|$ for all $n\geq 0$.

Let $\Gamma$ be the set of all connected affinoid opens $V\subseteq Y$ such that $V\to f(V)$ is finite \'etale. By Proposition \ref{pp etale local structure}, such opens $V$ cover $Y$. We make $\Gamma$ into a graph where $V$ and $V'$ are adjacent if $V\cap V'\neq\emptyset$. The graph $\Gamma$ is connected since $Y$ is. Therefore to find a~bound on the cardinality of $\Gamma$ it is enough to find an absolute bound on the degree of any vertex $V$ of $\Gamma$.

Let $V'_0$ and $V'_1$ be elements of $\Gamma$ with $f(V'_0)=f(V'_1)$ (call this subset $U'$) and such that $V'_0\cap V=V'_1\cap V\neq \emptyset$. Then $V'_i\to U'$ are connected finite \'etale coverings. Since $Y\to X$ is partially proper and in particular separated, the map $Y_{U'}\to U'$ is separated, and thus $V_0'\cap V_1'\to V_0'\times_{U'}V_1'$ is a closed immersion. Since this top map is a map between \'etale spaces over $U'$, it is \'etale and hence open. Therefore $V'_0\cap V'_1\to V'_0\times_{U'} V'_1\to U'$ is finite \'etale and non-empty. Since $V'_0$ and $V'_1$ are connected, we must have $V'_0=V'_1$. Indeed, since $V_0'\cap V_1'$ is finite \'etale over $U'$ and $V_i'$ for $i=0,1$ are finite \'etale over $U'$, the inclusion map $V_0'\cap V_1'\to V_i'$ for $i=0,1$ is finite, thus the image is clopen and therefore everything.

The set of possible values of $V'\cap V$ for varying $V'\in \Gamma$ is bounded by $2^{\kappa_0}$, and the set of possible values of $f(V')$ is bounded by $2^{\kappa_1}$ where $\kappa_1$ is the cardinality of $|X|$. By the previous paragraph, the degree of $V$ as a node of $\Gamma$ is bounded by $2^{\kappa_0+\kappa_1}$.
\end{proof}

\subsection{The \'etale bootstrap principle}
\label{ss:et-bootstrap}

In this final subsection, we observe that all \'etale morphisms of adic spaces can be refined by covers of a particularly simple form. 

\begin{lem}[{\cite[Lemma~2.2.8]{Huberbook}}] \label{lem:twotwoeight}
    Let $f\colon V\to X$ be an \'etale morphism of affinoid adic spaces. Then, there exists an affinoid open cover $X=\bigcup_{i\in I} U_i$, finite \'etale maps $Z_i\to U_i$, and open immersions $f^{-1}(U_i)\to Z_i$ over $U_i$.
\end{lem}

The desired refinement statement is then as follows.

\begin{prop} \label{prop:bootstrap}
    Let $X$ be an adic space and let $V\to X$ be an \'etale cover. Then $\{V\to X\}$ admits as a refinement an \'etale cover $\{V_{ij}\to X\}_{i\in I,j\in J_i}$ such that for all $i\in I$ and $j\in J_i$ one has a factorization
    \[
         V_{ij}\to W_i\to U_i\to X 
    \]
    such that the following holds 
    \begin{itemize}
        \item $\{U_i\to X\}_{i\in I}$ forms an open cover, with $U_i$ connected,
        \item the maps $W_i\to U_i$ are connected finite \'etale Galois covers with Galois group $G_i$,
        \item for every $i\in I$, the family $\{V_{ij}\to W_i\}_{j\in I_j}$ is a finite affinoid open cover, stable under the action of $G_i$.
    \end{itemize}
\end{prop}

\begin{proof}
If $X=\bigcup U_i$ is an affinoid open cover and the statement holds for $V\times_X U_i\to U_i$, then it also holds for $X$. We may therefore assume that $X$ is affinoid. Similarly, we may replace $V\to X$ with any \'etale surjection $V'\to X$ which factors through $V$. We may therefore assume that $V$ is the disjoint union of affinoids. Since $V\to X$ is open and $X$ is quasi-compact, the images of finitely many of the affinoids covering $V$ will cover $X$, and hence we may assume that $V$ is the disjoint union of a finite number of affinoids and hence is affinoid. 

Applying Lemma~\ref{lem:twotwoeight} to $V\to X$, we find an open cover $X=\bigcup U_i$ by connected affinoids and factorizations $V\times_X U_i\to Z_i\to U_i$ with $Z_i$ finite \'etale. Applying again the previous step, we may assume that $V$ and $X$ are affinoid, with $X$ connected, and that we have a factorization $V\to Z\to X$ with $Z\to X$ finite \'etale.

Let $V_1, \ldots, V_m$ be the connected components of $V$, and let $U_i\subseteq X$ be the image of $V_i$. For each $i$, let $Z_i \subseteq Z\times_X U_i$ be the connected component of $Z\times_X U_i$ containing the image of $V_i$. Then $Z_i\to U_i$ is a connected finite \'etale cover; let $W_i\to U_i$ be a Galois connected finite \'etale cover dominating $Z_i\to U_i$. If $G_i$ denotes the automorphism group of $W_i/U_i$, then for $g\in G_i$ let $V_{ig} = g^{-1}(V_i\times_{Z_i} W_i)$. We set $J_i=G_i$.

By construction, $V_i$ intersects every fiber of $Z_i\to U_i$, and hence $V_i\times_{Z_i} W_i$ intersects every fiber of $W_i\to U_i$. Since $G_i$ acts transitively on the fibers of $W_i\to U_i$, we see that the translates $V_{ig} = g^{-1}(V_i\times_{Z_i} W_i)$ cover $W_i$.
\end{proof}

The above result can be interpreted as saying that the big \'etale topology of adic spaces is generated by the Zariski topology and the finite \'etale topology. While we do not wish to make this statement entirely precise, let us note the following concrete and useful corollary.

\begin{cor} \label{boostrap corollary} 
    Let $X$ be an adic space and let $\mc{C}$ be a fibered category over $\Et_X$. Let $P$ be a~property of objects of $\mc{C}$ such that the following holds:
    \begin{enumerate}[(1)] 
        \item property $P$ is stable under base change,
        \item property $P$ can be checked on an affinoid open cover,
        \item property $P$ can be checked on a finite \'etale Galois cover of connected affinoids,
    \end{enumerate}
    Then, property $P$ can be checked on an arbitrary \'etale cover.
\end{cor}

\section{Geometric arcs}
\label{s:geom-arcs}

In this section we discuss the theory of geometric arcs and intervals.

\subsection{\texorpdfstring{Intervals in rigid $K$-spaces and AVC}{Intervals in rigid K-spaces and AVC}}
\label{ss:intervals-in-rig-sp}

We start by recalling some basic definitions concerning half-open and closed intervals.  As intervals rarely exist in non-Hausdorff spaces, we define the notions of intervals in rigid $K$-spaces using the universal separated quotient.

\begin{definition}
    Let $X$ be a rigid $K$-space. By an \emph{interval} in $X$ we mean a subspace \mbox{$\ell\subseteq [X]$}, equipped with an order relation $\leqslant$ for which there is a~monotonic homeomorphism to one of the following two model spaces: $[0,1]$ (in which case we call $\ell$ an \emph{arc}) or $[0,1)$ (in which case we call $\ell$ an \emph{ray}).
\end{definition}

By a \textit{parameterization} of an interval we mean a given monotonic homeomorphism with its model space. Given a parameterization $i \colon [0,1]\isomto \gamma$ of an arc $\gamma\subseteq [X]$ the points $i(0), i(1)\in \gamma$ are referred to as the \emph{left} and \emph{right endpoint} of $\gamma$, respectively. Similarly, if $i:[0,1)\isomto \rho$ is a parameterized ray in $X$, then we call $i(0)$ the \emph{(left) endpoint} of $\rho$.

A \emph{subinterval} (\emph{subarc}, \emph{subray}) $\ell'$ of an interval $\ell$ is an interval in $X$ contained in $\ell$ whose order relation agrees with that induced by $\ell$. For an interval $\ell$ and $a,b\in\ell$ with $a<b$, we write $[a,b]$ for the unique subarc of $\ell$ with left endpoint $a$ and right endpoint $b$. 

By a \emph{degenerate} arc we mean a point. Recall that a topological space is called \emph{arc connected} if for any two distinct points $x$ and $y$ there exists an arc $\gamma$ which has endpoints $x$ and $y$. We say that it is \emph{uniquely arc connected} if such an arc is always unique. 

\medskip

In \S\ref{ss:abundance-geometric-intervals} it will be useful to know that an arc in rigid $K$-curves have an oc neighborhood basis consisting of affinoids. To this end, we make the following observations.
\begin{itemize}
    \item If $X$ is taut then for any arc $\gamma$ in $X$ the subset $\sep_X^{-1}(\gamma)$ is closed and quasi-compact by  Proposition \ref{taut properties}~\ref{tau properties 2},
    \item $\sep_X^{-1}(\gamma)$ is connected (as any map $\sep_X^{-1}(\gamma)\to \{0,1\}$ must factor through $\gamma$),
    \item $\gamma$ cannot be a connected component of $[X]$ (since such a connected component has a~dense set (of classical) points whose removal does not disconnect the space).
\end{itemize}
Combining these, one sees that Proposition \ref{curves are good} yields the following.

\begin{prop} \label{arc basis}
    Let $X$ be a smooth and separated rigid $K$-curve and $\gamma$ an arc in $X$. Then, the set of affinoid open oc neighborhoods of $\gamma$ forms an oc neighborhood basis of $\gamma$.
\end{prop}

We now recall a result which is an underpinning of the theory of geometric coverings. 

\begin{thm}[{\cite[Theorem 3.2.1 and Corollary 4.3.3]{BerkovichSpectral}}] \label{thm:curves are arc connected}
    Let $X$ be a connected good rigid $K$-space. Then, the topological space $[X]$ is arc connected. Moreover, if $X$ is a~smooth and separated rigid $K$-curve then $X$ has a basis $\{U_i\}$ where each $U_i$ is affinoid and $[U_i]=\sep_X(U_i)$ is uniquely arc connected.
\end{thm}

One of the ways that arcs in rigid spaces will enter into our theory of geometric coverings is via the following `topological valuative criterion'.

\begin{definition} \label{def:AVC}
    Let $f\colon Y\to X$ be a map of rigid $K$-spaces and let $i\colon [0,1]\to [X]$ be a~parameterized arc. We say that $f$ \emph{satisfies the arcwise valuative criterion (AVC) with respect to $i$} if for every commutative square of solid arrows
    \begin{equation} \label{eqn:PVC-square} 
        \xymatrix{
            [0,1)\ar[d]\ar[r] & [Y]\ar[d]^{[f]} \\
            [0,1] \ar[r]_i \ar@{.>}[ur] & [X]
        }
    \end{equation}
    there exists a unique dotted arrow making the diagram commute. We say that $f$ satisfies the \emph{arcwise valuative criterion (AVC)} if it satisfies AVC with respect to every parameterized arc $i:[0,1]\to [X]$.
\end{definition}

\begin{rem} \label{rem:proper-AVC}
One can show that every proper map of topological spaces whose fibers are discrete satisfies AVC.
\end{rem}

\begin{rem} \label{rem:separated maps and AVC} 
Suppose that $f\colon Y\to X$ is map of taut rigid $K$-spaces. Then in any given diagram as in \eqref{eqn:PVC-square}, the dotted arrow is unique if it exists. Indeed, it suffices to note that since $X$ and $Y$ are taut, the spaces $[X]$ and $[Y]$ are Hausdorff, and hence $[f]\colon [Y]\to [X]$ is separated.
\end{rem}

\subsection{Definition of geometric intervals and basic operations}
\label{ss:geom-arcs-def}

In this subsection we define the notion of a geometric structure on an interval, provide a practical way of building such geometric structures, and finally provide some simple operations one can use to produce new such structures from old ones.

\medskip

\paragraph{Definition of geometric intervals} 

We start upgrading the notion of an interval to that of a `geometric interval', which is the synthesis of the notion of an interval and an \'etale path in algebraic geometry (see \S\ref{ss:geometric-points-and-fundamental-groups} for a recollection of notation).

\begin{definition} \label{def:geom-int}
    Let $X$ be a rigid $K$-space. A \emph{geometric interval} $\ov\ell$ in $X$ consists of the following data:
    \begin{itemize}
        \item an interval $\ell$ in $X$ (called the \emph{underlying interval} of $\ov\ell$),
        \item for every point $z\in \ell$, a geometric point $\ov{z}$ of $X$ anchored at $z^\mx$,
        \item for every subarc $[a,b]\subseteq \ell$ and  open oc neighborhood $U$ of $[a,b]$ an \'etale path
        \[
            \iota^U_{a,b} \in \pi_1^\alg(U; \ov a, \ov b),
        \]
    \end{itemize}
    such that the following conditions hold:
    \begin{enumerate}[(1)]
        \item for a subarc $[a,b]\subseteq \ell$ and two open oc neighborhoods $U, U'$ of $[a,b]$ such that $U\subseteq U'$, the map
        \[ 
            \pi_1^\alg(U; \ov a, \ov b) \to \pi^\alg_1(U'; \ov a, \ov b)
        \]
        induced by the inclusion $U\to U'$ maps $\iota^U_{a,b}$ to $\iota^{U'}_{a,b}$,
        \item for subarcs $[a,b]$ and $[b,c]$ of $\ell$, 
        and for every $U\subseteq X$ open oc neighborhood of $[a,c] = [a,b]\cup [b,c]$, the composition map
        \[ 
            \pi_1^\alg(U; \ov a, \ov b) \times \pi_1^\alg(U; \ov b, \ov c) \to \pi_1^\alg(U; \ov a, \ov c)
        \]
        maps $(\iota^U_{a,b}, \iota^U_{b,c})$ to $\iota^U_{a,c}$.
    \end{enumerate}
    We sometimes write $\iota^{U,\ov{\ell}}_{a,b}$ if we want to emphasize the role of $\ov{\ell}$. We also extend the definition of $\iota^U_{a,b}$ by $\iota^U_{a,a} = {\rm id}$ and $\iota^U_{b,a} = (\iota_{a,b}^U)^{-1}$. 
\end{definition}

We would like to say what it means for two geometric intervals to be `equivalent', which intuitively means that they give equivalent means of `parallel transport'.

\begin{definition} 
    Two geometric arcs $\ov\ell_1$ and $\ov\ell_2$ in the  rigid $K$-space $X$ are \emph{equivalent} if $\ell_1=\ell_2=:\ell$ and there exists equivalences $\ov{z}_1\to\ov{z}_2$ between the geometric points lying over each points $z$ of $\ell$ such that for all $a,b\in\ell$ the map
    \begin{equation*}
      \pi_1^\alg(U,\ov{a}_1,\ov{b}_1)\to \pi_1^\alg(U,\ov{a}_2,\ov{b}_2)
    \end{equation*}
    carries $\iota^{U,\ov{\ell}_1}_{a,b}$ to $\iota^{U,\ov{\ell}_2}_{a,b}$. 
\end{definition}

\medskip

\paragraph{Geometric structures defined on a basis} 
One may endow an interval with a~geometric structure by only specifying the \'etale paths $\iota_{a,b}^U$ for `sufficiently many $U$'. This is similar to the ability to define a sheaf by only specifying the values on a basis of open sets. To make this precise, let $X$ be a rigid $K$-space and let $\ell$ be an interval in $X$. Let us say that a set $\mc{U}$ of open subsets of $X$ is an \emph{$\ell$-basis} if for each point $z$ of $\ell$ the set $\mc{U}$ contains oc neighborhood basis of $z$.

\begin{definition} 
    Let $X$ be a rigid $K$-space and let $\ell$ be an interval in $X$. Let $\mc{U}$ be an $\ell$-basis. Then, a \emph{geometric $\mc{U}$-structure} consists of the data of a geometric point $\ov{z}$ anchored at $z^\mx$ for each point $z$ of $\ell$, and an \'etale path $\iota_{a,b}^U\in \pi_1^\alg(U;\ov a, \ov b)$ for every subarc $[a,b]\subseteq \ell$  and every oc open neighborhood $U$ of $[a,b]$ in $\mc{U}$, such that Axioms~1 and 2 of Definition~\ref{def:geom-int} are satisfied under the assumption that $U, U'\in \mc{U}$.
\end{definition}

We have the following result whose straightforward subdivision argument is omitted.

\begin{lem} \label{lem:partial-geom-arcs}
    Let $X$ be a rigid $K$-space, $\ell$ an interval in $X$, and $\mc{U}$ an $\ell$-basis. Then, every geometric $\mc{U}$-structure on $\ell$ extends uniquely to the structure of a geometric interval on $\ell$ (with the chosen geometric points).
\end{lem}

\medskip

\paragraph{Various operations on geometric intervals} 

In what follows $X$ will be a rigid $K$-space and $\ov{\ell}$ a geometric interval in $X$ with underlying interval $\ell$.

\begin{construction}[Subinterval] 
Let $\ell'$ be a subinterval of $\ell$. By considering only the geometric points $\ov{z}$ for $z\in\ell'$ and the \'etale paths $\iota^U_{a,b}$ for $[a,b]\subseteq \ell'$ one endows $\ell'$ with an induced geometric structure, denoted by $\ov{\ell}{}'$.
\end{construction}

\begin{construction}[Concatenation] 
Let $\ell$ be an interval in $X$ and let $z\in \ell$ be an interior point. Thus $\ell = \ell'\cup \ell''$ is a union of two intervals with $\ell\cap \ell'=\{z\}$ and $x\leq y$ for $x\in \ell'$ and $y\in \ell''$. Suppose that we are given structures of geometric intervals $\ov\ell{}'$ and $\ov\ell{}''$ for which the chosen geometric points $\ov{z}$ above $z$ agree. Then there exists a unique structure of a geometric interval $\ov\ell$ whose restrictions to $\ell'$ and $\ell''$ are the given ones. 
\end{construction}

\begin{construction}[Image] 
Let $f\colon X\to X'$ be a map of rigid $K$-spaces. Suppose that $[f]$ maps $\ell$ homeomorphically onto its image $\ell'$. For each point $z$ of $\ell$ denote $[f](z)$ by $z'$. Define the geometric point $\ov{z}'$ to be the image of the geometric point $\ov{z}$ (see \S\ref{ss:geometric-points-and-fundamental-groups}) under $f$. For each subarc $[a',b']\subseteq \ell'$ and each open oc neighborhood $U$ of $[a',b']$ we define $\iota^{U,\ov{\ell}'}_{a',b'}$ to be the image of $\iota^{f^{-1}(U),\ell}_{a,b}$ under the canonical map
\begin{equation*}
    \pi_1^\alg(f^{-1}(U); \ov{a},\ov{b})\to \pi_1^\alg(U;\ov{a'},\ov{b'})
\end{equation*}
We call this the \emph{image} of the geometric arc $\ov{\ell}$, and denote it by $f(\ov{\ell})$.
\end{construction}

\subsection{\texorpdfstring{$\mc{P}$-fields and their algebraic closures}{P-fields and their algebraic closures}}
\label{ss:p-fields}

In this subsection we develop theory about certain diagrams of fields that will play a role in showing that any two geometric points of a rigid $K$-curve can be connected by a geometric interval. A guiding example for this discussion is the following.

\begin{example} \label{etale paths on domains}
Let $R$ be an integral domain and let $x_i$ for $i=0,1$ be points of $\Spec(R)$. Set $\ov{R}$ to be the integral closure of $R$ in an algebraic closure $\ov{K}$ of its field of fractions. Let $\tilde x_i$ for $i=0,1$ be lifts of $x_i$ to $\Spec(\ov{R})$. The natural maps $\ov{y}_i\colon \Spec(k(\tilde x_i))\to \Spec(R)$ are geometric points anchored at $x_i$, and there is a natural isomorphism $F_{\ov{y}_0}\simeq F_{\ov{y}_1}$ as every finite \'etale cover of $\Spec(\ov{R})$ is split (see \stacks{0BQM}). As evidently $F_{\ov{x}_i}\simeq F_{\ov{y}_i}$ we see that the set $\pi_1^\alg(\Spec(R),\ov{x}_0,\ov{x}_1)$ is non-empty.
\end{example}

We would like to emulate this approach to create geometric structures on intervals. The notion of $\mc{P}$-fields plays the analogue of the generic point $\Spec(\mathrm{Frac}(R))$ in Example~\ref{etale paths on domains}. 

A (closed, possibly degenerate) interval in a totally ordered set $\mc{P}$ is a subset of the form 
\[ 
    [a,b] = \{x\in \mc{P} \,:\, a\leq x\leq b\}
\]
for some $a,b\in\mc{P}$ with $a\leq b$. We denote by $\cat{Int}(\mc{P})$ the poset of intervals in $\mc{P}$ ordered by inclusion. It is isomorphic as a poset to the set $\{(a, b)\in \mc{P}^2\,:\, a\leq b\}$ with $(a, b)\leq (a', b')$ if $a'\leq a\leq b\leq b'$. If $\varphi\colon \mc{Q}\to \mc{P}$ is a monotone map of totally ordered sets, we obtain a~monotone map $\varphi\colon \cat{Int}(\mc{Q})\to \cat{Int}(\mc{P})$ sending $[a,b]$ to $[\varphi(a), \varphi(b)]$.

\begin{definition} \label{def:p-field}
    Let $\mc{P}$ be a totally ordered set. A \emph{$\mc{P}$-field} is a functor
    \[
        \mc{K}\colon \cat{Int}(\mc{P})^{\rm opp} \to \cat{Fields}, \quad [a,b]\mapsto \mc{K}_{[a,b]}.
    \]
    A \emph{morphism} or \emph{extension} of $\mc{P}$-fields is a natural transformation $\mc{K}\to \mc{L}$. For $\mc{Q}\subseteq\mc{P}$, we denote by $\mc{K}|_Q$ the $\mc{Q}$-field obtained by restriction of $\mc{K}$ to $\cat{Int}(\mc{Q})^{\rm opp}$. 
\end{definition}

We say that an extension $\mc{K}\to \mc{L}$ \emph{algebraic} (resp.\ \emph{separable}) if for all $[a,b]\in \cat{Int}(\mc{P})$, the extension $\mc{K}_{[a,b]}\to \mc{L}_{[a,b]}$ is algebraic (resp.\ separable). We say that a $\mc{P}$-field $\mc{L}$ is \emph{algebraically closed} (resp.\ \emph{separably closed}) if all the fields $\mc{L}_{[a,b]}$ are algebraically (resp.\ separably) closed. We then have the following definition.

\begin{definition} \label{def:p-field-ac}
    Let $\mc{P}$ be a totally ordered set and $\mc{K}$ a $\mc{P}$-field. Then an \emph{algebraic (resp.\ separable) closure} of $\mc{K}$ is an extension of $\mc{P}$-fields $\mc{K}\to\mc{L}$ which is algebraic (resp.\ separable algebraic) and for which $\mc{L}$ is algebraically (resp.\ separably) closed.
\end{definition}

\begin{prop} \label{p-field algebraic closures exist}
    Let $\mc{P}$ be a totally ordered set. Then, every $\mc{P}$-field has an algebraic (resp.\ separable) closure.
\end{prop}

\begin{proof} 
By Lemma \ref{non-zero colim} below the ring $C=\colim (\iota\circ \mc{K})$ is non-zero. Therefore, it admits a~homomorphism $C\to L$ into an algebraically closed field $L$. Thus there exists compatible embeddings $\mc{K}_{[a,b]}\to C\to L$ for all $[a,b]\in \cat{Int}(\mc{P})$. If one sets $\ov{\mc{K}}_{[a,b]}$ to be the algebraic (resp.\ separable) closure of $\mc{K}_{[a,b]}$ in $L$, one sees that such fields are functorial, and so form an algebraic (resp.\ separable) closure $\ov{\mc{K}}$ of $\mc{K}$.
\end{proof}

\begin{lem} \label{non-zero colim} 
    Let $\mc{P}$ be a totally ordered set. Then, for every $\mc{P}$-field $\mc{K}$ the colimit $\colim (\iota\circ \mc{K})$ is a non-zero ring, where $\iota\colon \cat{Fields}\hookrightarrow \cat{Rings}$ is the natural inclusion.
\end{lem}

\begin{proof} 
Let us first observe that by \cite[Lemma 3.2.8]{KashiwaraSchapira}, the colimit can be rewritten as
\begin{equation*}
    \colim (\iota\circ \mc{K})  
    =\varinjlim_{J\in\mc{J}}\,\,\colim (\iota\circ \mc{K}|_J)
\end{equation*}
where $\mc{J}$ is any cofinal set of finite subposets of $\cat{Int}(\mc{P})^\mathrm{op}$. Since this outer colimit is filtered, and the filtered colimit of non-zero rings is non-zero, it suffices to find a cofinal set of finite subposets $J$ of $\cat{Int}(\mc{P})^{\mathrm{op}}$ such that $\colim (\iota\circ \mc{K}|_J)$ is non-zero.

To do this, for each finite subset $T$ of $\mc{P}$, let
\begin{equation*}
    Q_T=\left\{[a,b]\in\cat{Int}(\mc{Q})^\mathrm{op}: a,b\in T\right\}.
\end{equation*}
Note that $Q_T$ is finite, and that the set of $Q_T$ is cofinal in the set of all finite subposets of $\cat{Int}(\mc{P})^\mathrm{op}$, since for any finite subposet $J$ one has that $J\subseteq Q_T$ where $T$ comprises of the finitely many endpoints of elements of $J$.

To show that $\colim (\iota\circ K|_{Q_T})$ is non-zero, let $Z_T\subseteq Q_T$ be the set of all intervals $[a,b]\in Q_T$ such that $[a,b]\cap T = \{a,b\}$. Since for every $q\in Q_T$, the set $\{x\in Z_T\,:\, x\geq q\}$ of its upper bounds in $Z_T$ is a non-empty and connected poset, by \cite[Chapter IX, \S3, Theorem 1]{MacLane} we have
\[
    \colim (\iota\circ \mc{K}|_{Q_T}) = \colim (\iota\circ \mc{K}|_{Z_T}).
\]

If $T=\{x_0<x_1<\ldots<x_m\}$, then $Z_T$ consists of the intervals $[x_i, x_i]$ ($i=0, \ldots, m$) and $[x_{i-1}, x_i]$ ($i=1, \ldots, m$). The colimit over $Z_T$ is an iterated pushout (tensor product)
\begin{equation*}
    \colim (\iota\circ \mc{K}|_{Z_T}) = K_{[x_0, x_0]} \otimes_{K_{[x_0,x_1]}} K_{[x_1,x_1]} \otimes_{K_{[x_1, x_2]}} K_{[x_2, x_2]} \otimes\cdots \otimes_{K_{[x_{m-1}, x_m]}} K_{[x_m,x_m]}.
\end{equation*}
Since the tensor product of two non-zero algebras over a field is non-zero, we deduce that $\colim (\iota\circ \mc{K}|_{Z_T})\neq 0$.
\end{proof}

\begin{rem}
Not every diagram of fields indexed by a connected poset admits an ``algebraic closure'' in the sense similar to Definition~\ref{def:p-field-ac}, and such diagrams may have zero colimit. This shows the importance of restricting our attention to diagrams indexed by posets of the type $\cat{Int}(\mc{P})$.
\end{rem}

The following lemma will be used later for extending geometric rays to geometric arcs.

\begin{lem} \label{lem:extending-p-field}
    Let $\mc{P}$ be a totally ordered set with largest element $1$, and let $\mc{P}^\circ=\mc{P}\setminus\{1\}$. Let $\mc{K}$ be a $\mc{P}$-field, and let $\mc{K}^\circ$ be its restriction to $\mc{P}^\circ$. Let $\mc{K}^\circ\to \ov{\mc{K}}{}^\circ$ be an algebraic (resp.\ separable) closure of $\mc{K}^\circ$. Then there exists an algebraic (resp.\ separable) closure $\mc{K}\to\ov{\mc{K}}$ whose restriction to $\mc{P}^\circ$ coincides with $\mc{K}^\circ \to \ov{\mc{K}}{}^\circ$.
\end{lem}

\begin{proof}
We only treat algebraic closures, the proof for separable closures is the same. The intervals in $\cat{Int}(\mc{P})\setminus \cat{Int}(\mc{P}^\circ)$ are of the form $[a,1]$ with $a\in \mc{P}$. We first extend the definition of $\ov{\mc{K}}{}^\circ$ to such intervals with $a<1$ by defining $\ov{\mc{K}}_{[a,1]}$ to be the algebraic closure of $\mc{K}_{[a,1]}\subseteq \mc{K}_{[a,a]}=\mc{K}^\circ_{[a,a]}$ in $\ov{\mc{K}}{}^\circ_{[a,a]}$. 

We check that this naturally defines a functor on $\cat{Int}(\mc{P})^{\rm opp}\setminus \{[1,1]\}$. First, we note that $\ov{\mc{K}}_{[a,1]}$ coincides with the algebraic closure of $\mc{K}_{[a,1]}\subseteq \mc{K}_{[a,b]}$ in $\ov{\mc{K}}{}^\circ_{[a,b]}$ for any $a\leq b<1$. If $a\leq b<1$, then the commutative diagram
\[ 
    \xymatrix{
        \mc{K}_{[a,1]} \ar[r] \ar[d] & \mc{K}^{\circ}_{[a,b]} \ar[r] \ar[d] & \ov{\mc{K}}{}^{\circ}_{[a,b]} \ar[d] \\
        \mc{K}_{[b,1]} \ar[r] & \mc{K}^\circ_{[b,b]} \ar[r] & \ov{\mc{K}}{}^\circ_{[b,b]} 
    }
\]
produces a map $\ov{\mc{K}}_{[a,1]}\to \ov{\mc{K}}_{[b,1]}$. These maps are clearly functorial for $a\leq b\leq c<1$, and allow us to extend the functor $\ov{\mc{K}}{}^\circ$ to $\cat{Int}(\mc{P})^{\rm opp}\setminus \{[1,1]\}$.

To extend the functor to $[1,1]$, we let $\mc{K}_{[1^-, 1]}$ (resp.\ $\ov{\mc{K}}_{[1^-, 1]}$) be the colimit of $\mc{K}_{[a,1]}$ (resp.\ $\ov{\mc{K}}_{[a,1]}$) over all $a<1$. Then $\ov{\mc{K}}_{[1^-,1]}$ is an algebraic closure of $\mc{K}_{[1^-,1]}$. We pick an algebraic closure $\ov{\mc{K}}_{[1,1]}$ of $\mc{K}_{[1,1]}$ containing $\ov{\mc{K}}_{[1^-,1]}$. We then define the maps $\ov{\mc{K}}_{[a,1]} \to \ov{\mc{K}}_{[1,1]}$ by the composition $\ov{\mc{K}}_{[a,1]}\to \ov{\mc{K}}_{[1^-,1]}\to \ov{\mc{K}}_{[1,1]}$. This extends the functor $\ov{\mc{K}}$ and the natural transformation $\mc{K}\to \ov{\mc{K}}$ to all of $\cat{Int}(\mc{P})^{\rm opp}$ as desired.
\end{proof}

\subsection{Abundance of geometric intervals on curves}\label{ss:abundance-geometric-intervals}

In this section we verify that any interval is the underlying interval of a geometric interval. We establish this by two constructions. One will notice a similarity with the construction from Example~\ref{etale paths on domains}. 

Let $X$ be a smooth and separated rigid $K$-curve and let $\ell$ be an interval in $X$. Our first construction associates to $\ell$ an $\ell$-field, a sort of `meromorphic structure sheaf' for $\ell$.

\begin{construction} \label{cons:ell-field}
For any subarc $[a,b]$ of $\ell$, let us define the ring
\begin{equation*}
    \mc{R}_{[a,b]}=\varinjlim_{\substack{\text{open oc}\\ \text{neighborhood of } [a,b]}}\h_X(U)=\varinjlim_{\substack{\text{affinoid open}\\ \text{oc neighborhood of }[a,b]}}\h_X(U)
\end{equation*}
where the latter equality follows from Proposition \ref{arc basis}. 

\begin{lem} 
    The rings $\mc{R}_{[a,b]}$ are integral domains, and for $[a,b]\subseteq [a',b']$ the natural restriction maps $\mc{R}_{[a',b']}\to \mc{R}_{[a,b]}$ are injections.
\end{lem}

\begin{proof} 
Let us start by showing that $\mc{R}_{[a,b]}$ is a domain. By \cite[Chapter 0, Proposition~3.1.1(2)]{FujiwaraKato} it suffices to show that cofinal in the set of open affinoid oc neighborhoods of $\gamma$ are such open affinoid oc neighborhoods $U$ with $\h_X(U)$ an integral domain. But, cofinal in such affinoid domains are those $U$ which are (smooth and) connected, and for such $U$ the ring  $\h_X(U)$ is an integral domain (cf.\@ \cite[Corollary 3.3.21]{BerkovichSpectral}).
 
Let us now show that the map $\mc{R}_{[a',b']}\to \mc{R}_{[a,b]}$ for $[a,b]\subseteq [a',b']$ is injective. To see this let $f\in \h_X(U)$ be an element of $\mc{R}_{[a',b']}$ with zero image in $\mc{R}_{[a,b]}$, where $U$ is a~connected oc open neighborhood of $[a,b]$. Then, by definition there exists a connected oc open neighborhood $U'$ of $[a,b]$ contained in $U$ such that $f|_{U'}=0$. But, this implies that $V(f)\subseteq U$ contains an oc open subset. Therefore (cf.\@ loc.\@ cit.\@) that $V(f)=U$ and so $f$ equals zero in $\h_X(U)$ and so $f$ is zero in $\mc{R}_{[a,b]}$ as desired.
\end{proof}

Let us denote by $\mc{K}_{[a,b]}$ the fraction field of $\mc{R}_{[a,b]}$. By the above injectivity statement, we see that if $[a,b]\subseteq [a',b']$ then one gets an induced map $\mc{K}_{[a',b']}\to \mc{K}_{[a,b]}$. Thus, we see that one has a functor 
\[
    \mc{K}_\ell\colon \cat{Int}(\ell)^{\rm opp}\to \cat{Fields}, \quad [a,b]\mapsto \mc{K}_{[a,b]}
\]
and thus we obtain an $\ell$-field $\mc{K}_\ell$ which we call the $\ell$-field \emph{associated} to $\ell$ in $X$. This finishes our first construction. \hfill $\square$ {\it (Construction~\ref{cons:ell-field})}
\end{construction}

We now show that an algebraic closure of $\mc{K}_\ell$ gives rise to a~geometric structure on $\ell$.

\begin{construction} \label{geometric interval construction}
Let $\ov{\mc{K}}_\ell$ be an algebraic closure of $\mc{K}_\ell$. For $a$ in $\ell$, let us pick a~geometric point $\ov{a}$ of $X$ anchored at $a^\mx$ and let us abbreviate $\mc{R}_{[a,a]}$ (resp.\@ $\mc{K}_{[a,a]}$) to $\mc{R}_a$ (resp.\@ $\mc{K}_a$). The point $\ov a$ defines a geometric point on $\Spec(\mc{R}_a)$, denoted $\ov a$ as well. As in Example~\ref{etale paths on domains}, we pick an \'etale path $\iota_a \in \pi_1^\alg(\Spec(\mc{R}_a); \ov{a}, \tilde a)$ where $\tilde a$ is the geometric point $\Spec(\ov{\mc{K}}_a)\to \Spec(\mc{R}_a)$.

Suppose now that $[a,b]\subseteq \ell$ and that $U=\Spa(A)$ is an affinoid open oc neighborhood of $[a,b]$ in $X$. We would like to give an isomorphism of fiber functors:
\begin{equation*}
    \iota^U_{a,b}\colon F_{\ov{a}}\to F_{\ov{b}}
\end{equation*}
where $F_{\ov{a}}$ and $F_{\ov{b}}$ are the natural fiber functors on $\UFEt_U$. Note that we have a natural bijection (see the end of \S\ref{ss:geometric-points-and-fundamental-groups})
\begin{equation}\label{map of paths eq}
    \pi_1^\alg(\Spa(A);\Spa(k(\ov{a})),\Spa(k(\ov{b})))\isomto \pi_1^\alg(\Spec(A);\Spec(k(\ov{a})),\Spec(k(\ov{b})))
\end{equation}
But $\Spec(\ov{\mc{K}}_a)$ and $\Spec(\ov{\mc{K}}_a)$ are both geometric points of $\Spec(\ov{\mc{K}}_{[a,b]})$ as well as of $\Spec(A)$, which gives isomorphisms of fiber functors $j^U_{a,b}:F_{\ov{\mc{K}}_a}\simeq F_{\ov{\mc{K}}_b}$. 

We define $\iota_{a,b}^U$ to be the preimage under the map described in Equation \eqref{map of paths eq} of the following composition 
\begin{equation*}
    F_{\Spec(k(\ov{a}))}\xrightarrow{\iota_a}F_{\Spec(\ov{\mc{K}}_a)}\xrightarrow{j_{a,b}^U}F_{\Spec(\ov{\mc{K}}_b)}\xrightarrow{\iota_b^{-1}}F_{\Spec(k(\ov{b}))}
\end{equation*}
It is clear that $\iota_{a,b}^U$ doesn't depend on the affinoid open oc neighborhood $U$ of $[a,b]$. So then, for an open oc neighborhood $U$ of $[a,b]$ we can, as at the beginning of the proof, find an affinoid neighborhood $U'$ of $[a,b]$ contained in $U$. We then set $\iota^{U}_{a,b}$ to be the image of $\iota^{U'}_{a,b}$. As above indicated above, this definition is independent of the choice of $U'$.

The set of isomorphisms $\{\iota_{a,b}^U\}$, where now $[a,b]\subseteq \ell$ is arbitrary and $U$ is allowed to either be an affinoid or open neighborhood of $[a,b]$, is easily seen to satisfy the desired conditions. Thus, the $\iota^U_{a,b}$ together define a geometric interval $\ov\ell$ with underlying interval~$\ell$.
\end{construction}

Combining these two constructions we obtain the following.

\begin{prop} \label{geometric intervals exist} 
    Let $X$ be a smooth and separated rigid $K$-curve and let $\ell$ be an interval in $X$. Then, there exists a geometric interval $\ov{\ell}$ with underlying interval $\ell$.
\end{prop}

Moreover, often all geometric structures on an interval $\ell$ come from this construction.

\begin{prop} \label{exhaustion of geometric intervals} 
    Let $X$ be a smooth and separated rigid $K$-curve and let $\ell$ be an interval in $X$ whose endpoints are not classical points. Then, up to equivalence every geometric structure comes from Construction \ref{geometric interval construction}.
\end{prop}

\begin{proof}
Fix $[a,b]\subseteq\ell$ and let $\{U_i=\Spa(R_i)\}_{i\in I}$ be a cofinal family of affinoid oc neighborhoods of $[a,b]$. Then $\mc{R}_{[a,b]}= \varinjlim_{i\in I} R_i$, and hence
\[ 
    \pi_1^\alg(\Spec(\mc{R}_{[a,b]}); \ov a, \ov b) 
    = \varprojlim_{i\in I} \pi_1^\alg(\Spec(R_i); \ov a, \ov b)
    = \varprojlim_{i\in I} \pi_1^\alg(U_i; \ov a, \ov b),
\]
(see e.g.\ \cite[Remark 1.2.9]{KedlayaLiu}). The elements $\iota^{U_i}_{a,b}$ form an element of the right hand side by Axiom~1, and hence we obtain an element $\iota_{a,b}\in \pi_1^\alg(\Spec(\mc{R}_{[a,b]}); \ov a, \ov b)$. 

Since neither $a$ nor $b$ is classical and $X$ is smooth, the ring $\mc{R}_{[a,b]}$ is a field, i.e.\ $\mc{R}_{[a,b]}=\mc{K}_{[a,b]}$. Therefore $\iota_{a,b} \in \pi_1^\alg(\Spec(\mc{K}_{[a,b]}); \ov a, \ov b)$ corresponds to an identification of the algebraic closures of $\mc{K}_{[a,b]}$ in $k(\ov a)$ and $k(\ov b)$. We denote by $\ov{\mc{K}}_{[a,b]}$ this common value. 

If $a\leq b\leq c$ in $\ell$, then $\mc{K}_{[a,c]}\to \mc{K}_{[a,b]}\to k(\ov a)$ shows that $\ov{\mc{K}}_{[a,c]}$ is contained in $\ov{\mc{K}}_{[a,b]}$  as subfields of $k(\ov a)$. Similarly, $\ov{\mc{K}}_{[a,c]}\subseteq \ov{\mc{K}}_{[b,c]}$ as subfields of $k(\ov c)$. Moreover, Axiom~2 implies that the following diagram commutes
\[ 
    \xymatrix@R=.5em@C=.7em{
        k(\ov a) & & k(\ov b) & & k(\ov c)\\
        & \ov{\mc{K}}_{[a,b]}\ar[ul]\ar[ur] & & \ov{\mc{K}}_{[b,c]} \ar[ul]\ar[ur] \\
        & & \ov{\mc{K}}_{[a,c]} \ar[ul]\ar[ur]\ar@/^2em/[uull] \ar@/_2em/[uurr]
    }
\]
For $[a,b]\supseteq [a',b']$, we define $\ov{\mc{K}}_{[a,b]}\to \ov{\mc{K}}_{[a',b']}$ as the composition $\ov{\mc{K}}_{[a,b]}\to \ov{\mc{K}}_{[a,b']}\to \ov{\mc{K}}_{[a',b']}$, which is the same as the composition $\ov{\mc{K}}_{[a,b]}\to \ov{\mc{K}}_{[a',b]}\to \ov{\mc{K}}_{[a',b']}$. This way we obtain an algebraic closure $\ov{\mc{K}} = \{\ov{\mc{K}}_{[a,b]}\}$ of the $\ell$-field $\{\mc{K}_{[a,b]}\}$. It is straightforward to check that the geometric interval defined by this algebraic closure is equivalent to $\ov\ell$.
\end{proof}

\begin{rem}
In the above proof, we used the following observation: if $\ell$ is an interval in $X$, then only its endpoints can be classical points. To show this, since $X$ is good we may assume that $X$ is affinoid and by Theorem \ref{thm:curves are arc connected} we may assume that $[X]$ is uniquely arc connected. But, if $T$ is a uniquely arc connected space, $\ell\subseteq T$ is an arc, and $t\in \ell$ is an interior point, then $T\setminus \{t\}$ is disconnected. However, removing a classical point from a connected smooth curve does not make it disconnected (see \cite[Corollary~2.7]{HansenVanishing}). 
\end{rem}

Lastly, we combine Lemma~\ref{lem:extending-p-field} with Proposition~\ref{exhaustion of geometric intervals} to extend geometric structures from rays to arcs. 

\begin{cor} \label{extend} 
    Let $X$ be a smooth and separated rigid $K$-curve and let $i\colon [0,1]\to [X]$ be a~parameterized arc. Set $\gamma=i([0,1])$ and $\rho=i([0,1))$. Then, for any structure of a~geometric ray $\ov{\rho}$ there exists a structure of a  geometric arc $\ov{\gamma}$ on $\gamma$ whose induced structure on $\rho$ coincides with $\ov{\rho}$ up to equivalence. 
\end{cor}

\begin{proof}
It is enough to treat the geometric ray induced by $\ov\rho$ on $i([\frac 1 2, 1))$. This way, we may assume that no points of $\rho$ are classical points. By Proposition~\ref{exhaustion of geometric intervals}, we may assume that the geometric structure $\ov\rho$ comes from an algebraic closure $\ov{K}^\rho$ of the $[0,1)$-field $K^\rho=\{K_{[a,b]}\}$ on $\rho$. By Lemma~\ref{lem:extending-p-field}, the algebraic closure $\ov K{}^\rho$ extends to an algebraic closure $\ov K{}^\gamma$ of $K^\gamma$. The induced structure of a geometric arc $\ov\gamma$ on $\gamma$ has the required property.
\end{proof}

\section{Geometric coverings}
\label{s:geom-cov}

In this section, we define geometric coverings, show that they satisfy reasonable geometric properties, and then show that the category of such coverings is a tame infinite Galois category in the sense of \cite{BhattScholze}.

\subsection{Lifting of geometric intervals}
\label{ss:geom-arc-lift}

Before we define our notion of geometric covering we will need some background results about liftings of geometric intervals. We show that under very mild conditions any finite \'etale map has unique lifting of geometric arcs. Finally, we show that the property unique lifting of geometric arcs and AVC  (Definition~\ref{def:AVC}) are essentially equivalent for \'etale and partially proper morphisms of reasonable curves.

\begin{definition}
    Let $f\colon X'\to X$ be a map of rigid $K$-spaces and let $\ov\ell$ be a geometric interval in $X$. By a \emph{lifting} of $\ov\ell$ to $X'$ we shall mean a geometric interval $\ov{\ell}{}'$ in $X'$ such that $[f]$ maps $\ell'$ homeomorphically onto $\ell$ and $\ov\ell$ is equal to the geometric interval $f(\ov{\ell}{}')$.
\end{definition}

It is useful to note that with this definition, if $\ov{\ell}{}'$ is a lift of the geometric arc $\ov{\ell}$ then for each point $x'$ of $\ell'$ if we set $x=f(x')$ then the geometric point $\ov{x}'$ (afforded to us by the geometric structure on $\ell'$) is a lift of the geometric point $\ov{x}$ (with a similar comment). So, in particular, we see $\ov{x}'$ is a point of $Y_{\ov{x}}$ is a canonical way.

\begin{definition}\label{def:unique-lifting}
    A map $Y\to X$ of rigid $K$-spaces satisfies the property of \emph{unique lifting of geometric arcs} if for every geometric arc  $\ov\ell$ in $X$ with left geometric endpoint $\ov{x}$ and every lifting $\ov{x}'$ of $\ov{x}$ to $Y$, there exists a unique lifting of $\ov\ell$ to $Y$ with left geometric endpoint $\ov{x}'$.
\end{definition}

We now observe that unique lifting of geometric arcs implies unique lifting of geometric rays (defined in the analogous way).

\begin{prop} \label{automatic geometric ray lifting} 
    Let $Y\to X$ be a map of rigid $K$-spaces which has unique lifting of geometric arcs. Then, $Y\to X$ has unique lifting of geometric rays.
\end{prop}

\begin{proof} 
Let $\ov{\rho}$ be a geometric ray in $X$ with left geometric endpoint $\ov{x}$. Let us note then that for all $a<x$ in $\rho$ one has the induced geometric structure $\ov{\rho}_a$ on $\rho_a:=[a,x]$, and thus by the uniqueness of lifting of geometric arcs one has a unique lift of $\ov{\rho}_a$ to a geometric arc $\ov{\rho}_a'$ in $Y$. By uniqueness it is clear that if $a_1<a_2$ then the induced geometric arc structure from $\ov{\rho}_{a_1}'$ on the sublifting of $\rho_{a_2}$ must equal the lifting $\ov{\rho}_{a_2}'$. Therefore we can uniquely concatenate these to a lift of the geometric ray $\ov{\rho}$, which is clearly unique.
\end{proof}

We now show that one can lift geometric arcs uniquely along any finite \'etale morphism.

\begin{prop} \label{unique lifting finite etale}
    Let $Y\to X$ be a finite \'etale map of rigid $K$-spaces where $X$ is taut. Then $Y\to X$ satisfies unique lifting of geometric arcs.
\end{prop}

\begin{proof}
Let $\ov\gamma$ be a geometric arc in $X$, with geometric left endpoint $\ov x$, and let $\ov y\in Y_{\ov x}$ be a~lifting of $\ov x$. We denote by $x$ and $y$ the respective anchor points. For every $t\in\gamma$, consider the geometric point
\[
    \ov y(t) = \iota^X_{x,t}(\ov y) \in Y_{\ov t},
\]
and let $y(t)\in [Y]$ be the image of its anchor point. In particular, $\ov y(x)=\ov y$. 

We claim that the map $\gamma\to [Y]$ given by $t\mapsto y(t)$ is continuous. Since this map is clearly a~set theoretic lift of the map $\gamma\to [X]$ it is injective and hence its continuity implies its image is an arc $\gamma'_0$ in $[Y]$ mapping homeomorphically onto $\gamma$.

To verify that this map is continuous let $t_0\in \gamma$, let $y_0 = y(t_0)$, and let $U$ be an open neighborhood of $y_0$ in $[Y]$. Let us note that by Proposition \ref{[f] for etale and pp} applied to $[f]$ there exists an open neighborhood $V$ of $t_0$ such that the connected component $Y_0$ of $[f]^{-1}(V)$ containing $y_0$ is open and contained in $U$. Let us set $W=\sep_X^{-1}(V)$ and $C=\sep_Y^{-1}(Y_0)$. Thus $C$ is an overconvergent open subset of $Y$ which is a connected component of $Y_W$.

Let $[a,b]\subseteq \gamma$ be subarc containing $t_0$ in its interior and contained in $V$. By naturality of $\iota^{W}$ with respect to the morphism $C\to Y_W$ in $\FEt_W$, for every $t\in [a,b]$, the map $\iota^W_{t_0,t}\colon Y_{\ov t_0}\to Y_{\ov t}$ maps $C_{\ov t_0}$ into $C_{\ov t}$. By the axioms of a geometric structure, $\ov y(t) = \iota^W_{t_0, t}(\ov y(t_0))\in C_{\ov t}$, so we have $y(t)\in \sep_Y(C)=Y_0\subseteq U$ for $t\in [a,b]$, and so $t\mapsto y(t)$ is continuous.

Suppose that $\ov\gamma'$ is a lifting of $\ov\gamma$ with geometric left endpoint $\ov y$. Let $\mc{U}$ be the collection of open subsets $U'\subseteq Y$ such that $U'\to U=f(U')$ is finite \'etale. Then by Proposition \ref{pp etale local structure} the set $\mc{U}$ is a $\gamma$-basis. Let $[a',b']$ be a subarc of $\gamma'$ such that $[a',b']$ has some $U' \in \mc{U}$ as an open oc neighborhood and let $[a,b]$ be the corresponding subarc of $\gamma$. It follows  that $U=f(U')$ is an open oc neighborhood of $[a,b]$ and $f\colon U'\to U$ is finite \'etale. Let $Z'\to U'$ be a finite \'etale cover, let $Z=Z'\times_U U'$, and let $\Delta=(\mathrm{id},f) \colon Z'\to Z$ be the natural map. Then $Z$ is a finite \'etale cover of $U'$ and $\Delta$ is a morphism in the category $\FEt_{U'}$. We therefore have the commutative diagram
\begin{equation} \label{eqn:lift-geom-arc}
    \xymatrix{
        & Z'_{\ov a'} \ar[rr]^-{\iota_{a',b'}^{U'}(Z')} \ar[d]_\Delta & & Z'_{\ov b'} \ar[d]^\Delta \\
        Z'_{\ov a} \ar@{=}[r] \ar@/_2em/[rrrr]_-{\iota_{a,b}^U(Z'\to U)} & Z_{\ov a'} \ar[rr]^-{\iota_{a',b'}^{U'}(Z)}  & & Z_{\ov b'}  \ar@{=}[r] & Z'_{\ov b} \\
    }
\end{equation}
(the top square commutes by naturality of $\iota_{a',b'}^{U'}$, and the bottom one by definition of a~lifting of an arc). Since the vertical arrows are injective, this shows that the top map $\iota_{a',b'}^{U'}(Z')$ is uniquely determined by the diagram. By Lemma~\ref{lem:partial-geom-arcs}, this shows uniqueness of the system of \'etale paths $\iota_{a',b'}^{U'}$  given the underlying arc $\gamma'$ and lifting of geometric points $\ov a'\to \ov a$ for $a'\in \gamma'$ mapping to $a\in \gamma$. 

Moreover, applying the above to $U'=Z'=Y$, so that $Z=Y\times_X Y$, and the arc $[y, t']$ for some $t'\in \gamma'$ lying over $t\in\gamma$, we conclude that $\ov t' = \iota_{x,t}^X(Y)(\ov y) = \ov y(t)$ and $t' = y(t)$. Taking images in $[Y]$, we see that the underlying arc $\gamma'$ equals $\gamma'_0$ and that the liftings of geometric points of $\ov\gamma$ are also uniquely determined. 

It remains to construct the structure of a geometric arc on $\gamma'=\gamma'_0$ with $\ov t'= \ov y(t)$ as a chosen lift of the geometric point $\ov t$. To this end, again by Lemma~\ref{lem:partial-geom-arcs} it suffices to construct the maps $\iota_{a,b}^{U'}$ for $U'\in \mc{U}$ so that Axioms 1 and 2 of Definition~\ref{def:geom-int} are satisfied. For $[a',b']$, $U'\in \mc{U}$ and $U=f(U')\subseteq X$ as before we need to check that for every finite \'etale $Z'\to U'$ the map
\[ 
    \iota_{a,b}^U(Z'\to U) \colon Z'_{\ov a}\to Z'_{\ov b}
\]
maps $Z'_{\ov a'}$ to $Z'_{\ov b'}$. Since $Z'_{\ov a'}$ is the preimage of $\ov a'\in U_{\ov a}$ under $Z'_{\ov a}\to U_{\ov a}$, and similarly for $b$, this assertion follows from the functoriality of $\iota_{a,b}^U$ with respect to the morphism $Z'\to U'$ in $\FEt_U$. This defines the maps $\iota_{a',b'}^{U'}$. Axioms~1 and 2 of the definition of a geometric arc are easy to verify.
\end{proof}

Using this we can show that uniqueness of geometric arc liftings holds for arbitrary \'etale and partially proper maps.

\begin{cor} \label{unique lifting corollary} 
    Let $f\colon Y\to X$ be an \'etale and partially proper morphism of rigid $K$-spaces where $X$ is taut. Let $\ov{\gamma}$ be a geometric arc in $X$, and let $\ov x'$ be a lifting of the left geometric endpoint $\ov x$ of $\ov\gamma$ to $Y$. Then, a lifting of $\ov{\gamma}$ to $Y$ with left geometric endpoint $\ov x'$ is unique if it exists.
\end{cor}

\begin{proof} 
Let $\ov{\gamma}'_i$ ($i=0,1$) be two liftings of $\ov{\gamma}$ with left geometric endpoint $\ov x'$. Assume first that the underlying arcs $\gamma_i'$ agree, call this common arc $\gamma'$. By Proposition~\ref{pp etale local structure}, there exists a finite partition $\gamma' = \bigcup_{j=1}^m \gamma'_j$ and open oc neighborhoods $U_i$ of $\gamma'_j$ with the property that $f|_{U_j}\colon U_j\to f(U_j)$ is finite \'etale. Let $\gamma_j=[f](\gamma_j')$ be the induced subarcs of $\gamma$. By Proposition \ref{unique lifting finite etale} and induction on $j$ one has that the geometric structure on each $\gamma_j'$ induced by $\ov{\gamma}_i'$ must agree. This implies that $\ov{\gamma}'_1=\ov{\gamma}'_2$ as desired.
 
Suppose now that $\gamma'_i$ for $i=0,1$ are arbitrary lifts of $\ov{\gamma}$. Let $\delta'$ be the connected component $\gamma'_0\cap \gamma'_1$ containing the anchor point $x'$ of $\ov x'$. Since $X$ is taut and $f$ is partially proper we know that $Y$ is also taut by Proposition \ref{pp implies taut}, and thus $[Y]$ is Hausdorff by Proposition \ref{taut properties}~\ref{taut properties 1}. Thus, $\delta'$ is a (closed, possibly degenerate) subarc of both $\gamma'_0$ and $\gamma'_1$, mapping homeomorphically via $[f]$ onto a subarc $\delta$ of $\gamma$. Let $\ov{\delta}'_i$ ($i=0,1$) be the induced geometric structure on $\delta'$ inherited from $\ov{\gamma}_i'$. Then $\ov{\delta}'_i$ are two geometric lifts of $\ov{\delta}$ with the same underlying arc. Therefore, by the previous paragraph, we have $\ov{\delta}'_0=\ov{\delta}'_1$. 

Denote by $\ov{z}'$ common right geometric endpoint of $\ov{\delta}'_0=\ov{\delta}'_1$, lying over the right geometric endpoint $\ov{z}$ of $\delta$. If $z$ is the right endpoint of $\gamma$, we are done. In any case, by Proposition \ref{pp etale local structure} there exists an open oc neighborhood $U'$ of $z'$ such that $f\colon U'\to U$ (where $U=f(U')$) is finite \'etale.  If $[x,z]\neq \gamma$, then there exists an $a$ in the interior of $\gamma\cap \sep_X(U)$ with $a\geq z$ such that the preimages of $[z,a]$ in $\sep_Y(U)\cap \gamma'_i$ are subarcs $[z', a'_i]$ ($i=0,1$) with $a'_0\neq a'_1$. However, this contradicts Proposition~\ref{unique lifting finite etale} applied to the map $U'\to U$, the geometric subarc of $\gamma$ supported on $[z,a]$, and the lifting $\ov z'$ of $\ov z$ to $U'$.
\end{proof}

While the above pertains to geometric arcs, one can use the above in conjunction with Proposition \ref{geometric intervals exist} and Proposition \ref{automatic geometric ray lifting} to obtain the following useful topological corollary.

\begin{cor} \label{arc and ray lift} 
     Let $f\colon Y\to X$ be a surjective finite \'etale morphism where $X$ is a~smooth and separated rigid $K$-curve. Then, for any interval $\ell$ in $X$, there exist finitely many intervals $\ell_i$ ($1\leq i\leq m$) in $Y$ mapping homeomorphically onto $\ell$ such that
    \begin{equation*}
        [f]^{-1}(\ell)=\bigcup_{i=1}^m \ell_i.
    \end{equation*}
\end{cor}

\begin{proof}
First, assume that $\ell$ is an arc. Fix a geometric arc structure $\ov{\ell}$ on $\ell$ which exists by Proposition \ref{geometric intervals exist}. Let us also fix a parameterization $i\colon [0,1]\to [X]$ for $\ell$. Let $\ov{x}$ be the left geometric endpoint of $\ov{\ell}$ and consider the geometric fiber $Y_{\ov{x}}$. This contains $n$ geometric points of $Y$ where $n$ is the degree of $f$. Now, consider the  unique geometric arcs $\ov{\ell}_i$ starting from these geometric points and lifting $\ov{\ell}$. Obviously, we have $[f]^{-1}(\ell) \supset \bigcup_j \ell_j$. Conversely, let $w$ be a~point of $[f]^{-1}(\ell)$. Pick a geometric point $\ov{w}$ above $w$ that lifts the geometric point of $\ov{\ell}$ that has the maximal point in $X$ above $v=[f](w) \in \ell$ as its anchor point. We can lift $\ov{\ell}|_{[0,v]}$ and $\ov{\ell}|_{[v,1]}$ uniquely to geometric arcs in $Y$ having $\ov{w}$ as a point. The concatenation of the two, however, produces a geometric arc in $Y$ lifting $\ov{\ell}$. This shows $[f]^{-1}(\ell) \subset \bigcup_j \ell_j$.

To prove the geometric ray case, we write the ray as an increasing union of arcs, put geometric arc structures on those arcs in a compatible way and apply the above.
\end{proof}

We next show that AVC is strong enough to imply unique lifting of geometric arcs for \'etale and partially proper morphisms. 

\begin{prop} \label{pre-geometric covering lifts geometric arcs}
    Let $f\colon Y\to X$ be a map of rigid $K$-spaces where $X$ is assumed taut, and let $\ov{\gamma}$ be a geometric arc in $X$ with underlying parameterized arc $i\colon [0,1]\to [X]$. Suppose that
    $f$ is \'etale, partially proper, and $[f]$ satisfies AVC with respect to every parameterized subarc of $i\colon [0,1]\to [X]$. Then $f$ satisfies unique lifting with respect to $\ov{\gamma}$.
\end{prop}

\begin{proof}
In this proof, we tacitly use the uniqueness result of Corollary \ref{unique lifting corollary} several times. Let $\ov{x}$ be an endpoint anchored at $i(0)$, and let $\ov{x}'$ be a lift of $\ov{x}$ to $Y$. For each $t$ in $[0,1]$ let $\ov{\gamma}_t$ be the induced geometric arc structure to $i([0,t])$. Let $S$ be the the set of all $t \in [0,1]$ for which there exists a lifting  $\ov{\widetilde{\gamma}}_t$ of $\ov\gamma_t$ to a geometric arc in $Y$ with left endpoint $\ov{x}'$. Note that $0\in S$ and if $t\in S$ then $[0,t]\subseteq S$. Thus, to show that $S=[0,1]$ it suffices to show that 
\begin{itemize}
    \item if $t\in S$ for all $t<t_0$ then $t_0\in S$,
    \item if $1\ne t_0\in S$ then $t_0+\varepsilon\in S$ for some $\varepsilon>0$.
\end{itemize} 
To see the first claim let us then note that the union of the underlying arcs $\widetilde{\gamma}_t$ for $t<t_0$ lifting $\gamma_t$ is precisely a lift of the parameterized ray $i|_{[0,t_0)}$. By our assumption on AVC relative to parameterized subarcs of $i$ we can extend lift of the parameterized arc $i|_{[0,t_0]}$. Let $y'$ be the point in our lift living over $y= i(t_0)$. Note that by Proposition \ref{pp etale local structure} there exists an open oc neighborhood $U$ of $y'$ such that $f\colon U\to f(U)$ is finite \'etale. Since $f$ is \'etale and partially proper we know that $y$ lies in $\interior_X(f(U))$, and thus $i([0,t_0])\cap \sep_X(V)$ contains $i([t_1,t_0])$ for some $t_1<t_0$. But, since $f\colon U\to f(U)$ is finite \'etale, we know by Proposition \ref{unique lifting finite etale} that there exists a unique lift of the geometric arc structure on $i([t_1,t_0])$ to a geometric arc of $U$. This clearly glues to give a geometric lift of $\ov{\gamma}_{t_0}$ as desired.

To show the second claim let $y=i(t_0)$ and let $\ov{\widetilde{\gamma}}_{t_0}$ be our chosen lift of $\ov{\gamma}_{t_0}$. Let $y'$ be the point of $\widetilde{\gamma}_{t_0}$ lying over $y$. Again by Proposition \ref{pp etale local structure} there exists open oc neighborhoods $U$ of $y'$ such that $f\colon U\to f(U)$ is finite \'etale. Since $f$ is \'etale and partially proper we have that $y\in \interior_X(f(U))$ and thus we have that $i([0,1])\cap \sep_X(f(U))$ contains $i([t_1, t_0+\varepsilon])$ for some $\varepsilon>0$. One then proceeds as in the first case.
\end{proof}

Finally, we show that the converse to this result holds in the case of curves.

\begin{prop}\label{geometric covering equiv} 
    Let $f\colon Y\to X$ be an \'etale and partially proper morphism where $X$ is a~smooth and separated rigid $K$-curve. Then, $[f]$ satisfies AVC if and only if $f$ has unique lifting of geometric arcs.
\end{prop}

\begin{proof} 
If $f$ satisfies AVC then we have already shown in Proposition \ref{pre-geometric covering lifts geometric arcs} that $f$ satisfies unique lifting of geometric arcs. Conversely, let $i\colon [0,1]\to [X]$ be an arc and let $s\colon [0,1)\to [Y]$ be a~lifting of the ray. Set $\rho'=s([0,1))$ and $\rho=i([0,1))$. Let us note that since $Y$ is a smooth and separated rigid $K$-curve we can by Proposition \ref{geometric intervals exist} find a geometric structure $\ov{\rho}'$ on $\rho'$. Since $[f]$ maps $\rho'$ homeomorphically onto $\rho$ we note then that $f(\ov{\rho}')$ is a~geometric structure on $\rho$. By Corollary \ref{extend} we can extend this to a geometric structure $\ov{\gamma}$ on $\gamma$. We then note that by unique lifting of geometric arcs that there exists a lift $\ov{\gamma}'$ of $\ov{\gamma}$. 

We claim that $\gamma'$ is an arc lift of $\gamma$ extending the ray lift $\rho'$ of $\rho$. To see this, for each $x<1$ set
\begin{itemize}[leftmargin=*]
    \item $\gamma_x$ to be the arc $i([0,x])$, and $\ov{\gamma}_x$ to be the geometric structure on $\gamma_x$ induced by $\ov{\gamma}$,
    \item $\gamma'_x$ to be the subarc of $\gamma'$ lifting $\gamma_x$, and $\ov{\gamma}'_x$ the geometric structure on $\gamma'_x$ induced by $\ov{\gamma}'$,
    \item $\rho_x$ to be the subarc $s([0,x])$ of $\rho$, and $\ov{\rho}_x$ the geometric structure on $\rho_x$ induced by $\ov{\rho}$.
\end{itemize}
Note then that $\ov{\rho}_x$ and $\ov{\gamma}'_x$ are both lifts of the geometric arc $\ov{\gamma}_x$. By previous comment we know that $f$ satisfies uniqueness of liftings of geometric arcs, and thus $\ov{\rho}_x=\ov{\gamma}'_x$ and so, in particular, we see that 
\begin{equation*}
    \rho=\bigcup_{x>0}\rho_x=\bigcup_{x>0}\gamma'_x
\end{equation*}
which implies that $\gamma'$ is, indeed, an arc lift of $\gamma$ extending $\rho$ as desired.
\end{proof}

\subsection{The definition of geometric coverings and basic properties}
\label{ss:geom-cov-def}

We now define geometric coverings, which in essence are \'etale and partially proper maps which satisfy unique lifting of geometric arcs.  As the theory of geometric arcs and their lifts only works well on smooth and separated rigid $K$-curves (see Remark~\ref{rem:tame-paths}) we shall have to use such spaces as the basis of our test objects. This is made possible by the result in  Appendix~\ref{curve-connectedness appendix}.

By Proposition \ref{geometric covering equiv} we may also think of geometric coverings as \'etale maps which which satisfy both an algebro-geometric valuative criterion (in the form of being partially proper) but also a topological valuative criterion (in the form of satisfying AVC). Again, this requirement will be required to hold over smooth and separated rigid $K$-curves.

\begin{definition} \label{def:test-curve}
    Let $X$ be a rigid $K$-space. A \emph{test curve} in $X$ is the data of a a non-archimedean field extension $L$ of $K$, a smooth and separated $L$-curve $C$, and a map $C\to X$ of adic spaces over $K$.
\end{definition}

We note that if $Y\to X$ is \'etale and partially proper, and $C\to X$ is a test curve in $X$, then $Y_C\to Y$ is a test curve in  $Y$.

We now define our notion of geometric coverings. The equivalence of the two conditions follows from Proposition~\ref{geometric covering equiv}.

\begin{definition} \label{def:geom-cov}
    A morphism $Y\to X$ of rigid $K$-spaces is a \emph{geometric covering} if it is \'etale, partially proper, and satisfies one of the following two equivalent properties:
    \begin{enumerate}[(a)]
        \item for all test curves $C\to X$ the map $Y_C\to C$ has unique lifting of geometric arcs (Definition~\ref{def:unique-lifting}),
        \item for all test curves $C\to X$ the map $Y_C\to C$ satisfies AVC (Definition~\ref{def:AVC}).
    \end{enumerate}
    We denote the full subcategory of $\Et_X$ consisting of geometric coverings of $X$ by $\Cov_X$. 
\end{definition}

\begin{rem} 
Note that by Remark~\ref{rem:separated maps and AVC} and Corollary \ref{unique lifting corollary} the uniqueness part of (a) and (b) above are automatic.
\end{rem}

We now show that geometric coverings are closed under many natural operations, and are therefore quite abundant. The first substantive class is singled out by Proposition~\ref{unique lifting finite etale}.

\begin{prop} \label{finite etale are geometric coverings} 
    Let $Y\to X$ be a finite \'etale morphism of rigid $K$-spaces. Then, $Y\to X$ is a geometric covering.
\end{prop}

\begin{rem} 
One can also prove that finite \'etale maps satisfy AVC without the need for geometric arcs, using the corresponding fact about proper continuous maps with discrete fibers (see Remark~\ref{rem:proper-AVC}).
\end{rem}

Therefore the category $\FEt_X$ is a full subcategory of $\Cov_X$. By the result below, it also contains the category $\UFEt_X$ of disjoint unions of rigid $K$-spaces finite \'etale over $X$.  

\begin{prop} \label{disjoint union geometric coverings} 
    Let $X$ be a rigid $K$-space and let $\{Y_i\to X\}$ be a collection of morphisms of rigid $K$-spaces. Let us set $Y=\coprod_i Y_i$ with its natural map $Y\to X$. Then, $Y\to X$ is a~geometric covering if and only if $Y_i\to X$ is a geometric covering.
\end{prop}

In particular, one sees that for a morphism $Y\to X$ of rigid $K$-spaces to be a geometric covering it is necessary and sufficient to check that for every connected component $Y_i$ of $Y$ that $Y_i\to X$ is a geometric covering. 

The most difficult property of $\Cov_X$ to prove is the following.

\begin{prop} \label{etale local}
    The property of being a geometric covering is \'etale local on the target.
\end{prop}

Since $\Cov_X$ contains $\UFEt_X$ the \'etale local nature of geometric coverings implies that $\Cov_X$ in fact contains the \'etale stackification of $\UFEt_X$. This is significant since it implies that $\Cov_X$ contains not only the covering spaces considered in \cite{deJongFundamental} but also many natural generalizations of those covering spaces (see \cite[\S 3]{ALY1P2}).

To prove Proposition \ref{etale local} we will require a few intermediary results which are of independent interest. 

\begin{prop} \label{composition and base change geometric covering}
    Let $X$ be a rigid $K$-space. Then, the following statements are true.
    \begin{enumerate}[(a)] 
        \item \label{composition and base change geometric covering 1} {\rm (Composition)} If $f\colon Y\to X$ and $g\colon Z\to Y$ are geometric coverings, then $g\circ f$ is a~geometric covering. 
        \item \label{composition and base change geometric covering 2} {\rm (Pullback)} If $Y\to X$ is a geometric covering and $X'\to X$ is any morphism of rigid $K$-spaces, then $Y_{X'}\to X'$ is a geometric covering.
        \item {\rm (Change of base field)} If $Y\to X$ is a geometric covering and $L$ is a non-archimedean extension of $K$, then $Y_L\to X_L$ is a geometric covering.
        \item \label{geom covering cancellation} {\rm (Cancellation)} Let $X$ be a rigid $K$-space, $Y\to X$ be a geometric covering, and $Y'\to X$ be separated and \'etale morphism. Then, any $X$-morphism $Y\to Y'$ is a~geometric covering.
    \end{enumerate}
\end{prop}

\begin{proof} 
To see (a) we note that since the properties of \'etale and partially proper are preserved under composition, it suffices to show that for any test curve $C\to X$ in $X$ that the map $Z_C\to C$ satisfies unique lifting of geometric arcs. But, note that since $Z_C=Z_{Y_C}$, and $Y_C$ is a~test curve, one quickly shows that since geometric arcs may be uniquely lifted along $Z_{Y_C}\to Y_C$ and $Y_C\to C$, that they may be uniquely lifted along $Z_C\to C$.

To see (b), we note that since $Y_{X'}\to X'$ is \'etale and partially proper, it suffices to show that for all test curves $C\to X'$ in $X'$ one has that $(Y_{X'})_C\to C$ satisfies lifting of geometric arcs. But, since $(Y_{X'})_C$ is canonically isomorphic to $Y_C$ as $C$-spaces, the claim easily follows. The proof of (c) is identical.

For (d), note that by the cancellation principle\footnote{Vakil's `cancellation principle' \cite[10.1.19]{Vakilfoag} is the following: given a category $\mc{C}$ with pullbacks and a property $P$ of maps in $\mc{C}$ closed under pullbacks and composition, for every two maps $f\colon Y\to X$, $g\colon Z\to Y$ such that $f\circ g\colon Z\to X$ and the diagonal $\Delta_{Y/X}\colon Y\to Y\times_X Y$ have $P$, also $g$ has $P$} it suffices to check that $\Delta_{Y'/X}\colon Y'\to Y'\times_X Y'$ is a geometric covering. But, since $Y'\to X$ is separated and \'etale, we know that $\Delta_{Y'/X}$ is a clopen embedding (see \cite[Proposition 1.6.8]{Huberbook}) and so is a~geometric covering by Proposition \ref{finite etale are geometric coverings}.
\end{proof}

The other result we shall need to prove Proposition \ref{etale local} is the following result which, colloquially, says that the category $\Cov_X$ is closed under taking images.

\begin{prop} \label{image of geometric covering is geometric covering}
    Let $X$ be a rigid $K$-space and let $Y\to Y'$ be surjective map of rigid $X$-spaces. Assume that $Y\to X$ is a geometric covering and that  $Y'\to X$ is separated and \'etale. Then, $Y'\to X$ is a geometric covering.
\end{prop}

\begin{proof}
We note that $Y' \to X$ is partially proper, by Lemma \ref{image of partially proper}. It remains to check that for any test curve $C\to X$ in $X$ that the map $Y'_C\to C$ satisfies AVC. We may assume that $C=X$. Let $i\colon [0,1] \to [X]$ be an arc and $s'\colon [0,1) \to [Y']$ be a lift of $[0,1)$. By Proposition~\ref{composition and base change geometric covering}~\ref{geom covering cancellation} we know that $Y \to Y'$ is a geometric covering of smooth and separated rigid $K$-curves. By combining Proposition \ref{geometric intervals exist} and Proposition \ref{automatic geometric ray lifting} we see that one can, in particular, lift topological rays along $[Y'] \to [Y]$. Thus, one can lift $s'$ to a lift $s\colon [0,1) \to [Y]$. But, since $Y\to X$ is a~geometric covering we know that AVC holds for $[Y] \to [X]$, and thus we can lift $i$ to an arc on $[Y]$ extending $s$. The composition of this lift with the map $[Y]\to [Y']$ is a~lifting of $i$ to $[Y']$ extending $s'$ as desired.
\end{proof}

\begin{proof}[Proof of Proposition \ref{etale local}] 
By Corollary \ref{boostrap corollary} we reduce to showing that if $Y\to X$ is a~morphism of rigid $K$-spaces, then whether this map is a geometric covering can be checked on a~finite \'etale Galois cover and can be checked on an affinoid open cover. By Proposition \ref{prop:etale-etale-local} and Proposition \ref{prop:pp-etale-local} we know that $Y\to X$ is an \'etale and partially proper map, and it remains to check that AVC holds for the base change of $Y$ to any test curve $C\to X$. 
 
The claim concerning Galois covers follows easily from previously proven facts. Indeed, denote by $X' \to X$ the Galois cover such that $Y_{X'}\to X'$ is a geometric covering. Since $X'\to X$ is a~geometric covering by Proposition \ref{finite etale are geometric coverings} we see from Proposition \ref{composition and base change geometric covering} \ref{composition and base change geometric covering 1} that the composition $Y_{X'}\to X$ is a geometric covering. The map $Y_{X'} \to Y$ is surjective, and since $Y\to X$ is \'etale and partially proper we see by Proposition \ref{image of geometric covering is geometric covering} that $Y\to X$ is a geometric covering as desired.

To see the claim concerning affinoid open covers, let $\{U_j\}_{j\in J}$ be an affinoid open cover of $X$ such that $Y_{U_j}\to U_j$ is a geometric covering for all $j$. Taking an affinoid refinement of $\{U_j\times_X C\}$ and replacing $X$ with $C$, we are then reduced to showing that if $X$ is a smooth and separated curve and $Y\to X$ is an \'etale and partially proper map, whether this map satisfies AVC can be checked on an affinoid open cover. Since for any arc $\gamma\subseteq [X]$, its preimage $\sep_X^{-1}(\gamma)$ is quasi-compact, we may also assume that $X$ is quasi-compact and the index set $J$ is finite.

Let us note that for every arc $\gamma$ in $X$ and every affinoid open $U \subseteq X$, the intersection $[U] \cap \gamma$ is closed (because $[U]$ is closed in $[X]$) and has finitely many connected components (by Lemma~\ref{lem:finite-pi0} below). Let $i\colon [0,1]\to [X]$ be a parameterized arc, and let $s\colon [0,1)\to [Y]$ be a lifting of $i|_{[0,1)}$. Suppose first that $i([0,1])\subseteq [U_j]$ for some $j\in J$. Since $[Y_{U_j}]\to [U_j]$ satisfies AVC, we can extend $s$ to a map $[0,1]\to [Y_{U_j}]$. Since $[Y_{U_j}] = [Y]\times_{[X]} [U_j]$ by Proposition \ref{commuting sep and inverse image}, we see that $[Y]\to [X]$ satisfies AVC with respect to $i$. 

In any case, $[0,1]$ is covered by the finitely many closed subsets $i^{-1}([U_j])$, and each of them has finitely many connected components, so one of them contains $[1-\varepsilon,1]$ for some $\varepsilon>0$. In order to check AVC for a parameterized arc $i$, we may replace it with its restriction to $[1-\varepsilon,1]$ for any $\varepsilon>0$, and then we conclude by our first observation.
\end{proof}

The following lemma, used in the above proof, is the main reason we only use arcs in test curves in our theory of geometric coverings. Indeed, the analogous assertion is false in higher dimensions. See Remark~\ref{rem:tame-paths} for a discussion.

\begin{lem} \label{lem:finite-pi0}
    Let $X$ be smooth and separated rigid $K$-curve. Then, for every arc $\gamma$ in $X$ and every affinoid open $U \subseteq X$, the intersection $[U] \cap \gamma$ is closed and has finitely many connected components.
\end{lem}

\begin{proof}
By Proposition~\ref{curves are good}, the adic space $X$ is good, and so since $\gamma$ is quasi-compact and $X$ is separated we may assume that $X$ is affinoid.

By \cite[Lemma 3.1]{deJongFundamental} the Shilov boundary of $[U]$ is finite, and by combining this with \cite[Corollary 2.5.13.(ii) and Proposition 3.1.3]{BerkovichSpectral} it follows that the (topological) boundary $\partial_{[X]}[U]$ of $[U]$ in $[X]$ is finite. We have tacitly used here that the the Shilov boundary of $[U]$ matches the relative boundary $\partial([U]/\mathcal{M}(k))$ (see \cite[Example 3.4.2.5]{TemkinBerkovich}). As $\gamma$ and $[U]$ are closed in $[X]$, we have that $\partial_\gamma(\gamma \cap [U])$ is contained in $\partial_{[X]}([U])\cap \gamma$. It follows that $\gamma \cap [U]$ can be identified with a closed subset of $[0,1]$ with a finite boundary. But such a subset must be a finite union of closed subintervals, thus the claim follows.
\end{proof}

We finally mention some miscellaneous geometric properties of geometric coverings which one expects from a good analogue of the topological theory of covering spaces.

\begin{prop} \label{image of geometric covering clopen}
    Let $X$ be a rigid $K$-space and let $f\colon Y\to X$ be a geometric covering. Then, $f(Y)$ is a clopen subset of $X$. In particular, if $X$ is connected and $Y$ is non-empty then $f$ is surjective.
\end{prop}

\begin{proof} 
By considering the connected components of $X$, it is enough to show that for $X$ connected and $Y$ non-empty, we have $f(Y)=X$. Suppose now that $f$ is not surjective; we claim that in this case there exists a maximal point $x\notin f(Y)$. Indeed, since $f(Y)$ is an overconvergent open by Proposition \ref{[f] for etale and pp}, we have $f(Y)=\sep_X^{-1}(\sep(f(Y))$. By Corollary \ref{cor: general path connected} one can connect any maximal point $y$ in $f(Y)$ to $x$ by a sequence of connected test curves. By possibly changing $x$, we may thus assume that there exists a connected test curve $C\to X$ which contains points of $f(Y)$ and its complement. We may therefore assume that $X$ is a connected smooth and separated curve.

By Theorem~\ref{thm:curves are arc connected} and Proposition~\ref{geometric intervals exist} there exists a geometric arc $\ov{\gamma}$ with left endpoint $y$ and right endpoint $x$. Let $\ov{y}'$ be a lifting of the geometric point $\ov{y}$ above $y$. Then, by definition of a geometric covering there exists a lift $\ov{\gamma}'$ of $\gamma$. In particular, one of the endpoints of $\ov{\gamma}'$ lifts the point $x$, which is a contradiction.
\end{proof}

\begin{prop} \label{fibered product of geometric coverings}
    Let $X$ be a rigid $K$-space and suppose that $Y_i\to X$ for $i=1,2$ and $Z\to X$ are geometric coverings. Then, for any $X$-morphisms $Y_i\to Z$ for $i=1,2$ the map $Y_1\times_Z Y_2\to X$ is a geometric covering.
\end{prop}

\begin{proof} 
First, $Y_1\to Z$ is a geometric covering by Proposition~\ref{composition and base change geometric covering}~\ref{geom covering cancellation}. By Proposition~\ref{composition and base change geometric covering}~\ref{composition and base change geometric covering 2}, so is $Y_1\times_Z Y_2\to Y_2$. But, since $Y_2\to X$ is a geometric covering we know that the composition $Y_1\times_Z Y_2\to X$ is a geometric covering by Proposition \ref{composition and base change geometric covering} ~\ref{composition and base change geometric covering 1}.
\end{proof}

\subsection{A brief recollection of tame infinite Galois categories}\label{igc sec}

In the next section we show that $\Cov_X$, with any of its natural fiber functors $F_{\ov{x}}$, is a tame infinite Galois category in the sense of \cite[\S7]{BhattScholze}.\footnote{A similar formalism, under different names, appears also in \cite{LepageThesis}.} Thus, we shall briefly recall the setup of this theory to set notation and prime ourselves for the proof ahead.

In essence, the theory of tame infinite Galois categories seeks to axiomatize the study of the category $G\text{-}\cat{Set}$ of discrete sets with a continuous action of a topological group.  It is in analogy with the classical theory of Galois categories, where one studies finite sets with a continuous action of a profinite group. However, unlike the classical case, one needs an extra `tameness' condition to get a good theory (see e.g.\ \cite[Example 7.2.3]{BhattScholze}).

\begin{definition}[{\cite[Definition 7.2.1]{BhattScholze}}] 
    Let $\mc{C}$ be a category and $F\colon \mc{C}\to\cat{Set}$ be a functor (called the \emph{fiber functor}). We then call the pair $(\mc{C},F)$ an \emph{infinite Galois category} if the following properties hold:
    \begin{enumerate}[leftmargin=2.5cm]
        \item[\textbf{(IGC1)}] The category $\mc{C}$ is cocomplete and finitely complete.
        \item[\textbf{(IGC2)}] Each object $X$ of $\mc{C}$ is a coproduct of categorically connected objects of $\mc{C}$.
        \item[\textbf{(IGC3)}] There exists a set $S$ of connected objects of $\mc{C}$ which generates $\mc{C}$ under colimits.
        \item[\textbf{(IGC4)}] The functor $F$ is faithful, conservative, cocontinuous, and finitely continuous.
    \end{enumerate}
    We say that $(\mc{C},F)$ is \emph{tame} if for every categorically connected object $X$ of $\mc{C}$ the action of $\pi_1(\mc{C},F)$ on $F(X)$ is transitive. The \emph{fundamental group of $(\mc{C},F)$}, denoted $\pi_1(\mc{C},F)$ is the group $\Aut(F)$ endowed with the compact-open topology.\footnote{More precisely, for each $s$ in $S$, where $S$ is as in \textbf{(IGC3)}, we endow $\Aut(s)$ with the compact-open topology. We then endow $\Aut(F)$ with the subspace topology inherited from the natural map $\Aut(F)\to\prod_{s\in S}\Aut(s)$.} 
\end{definition}

In the above we used the terminology that an object $Y$ of a category $\mc{C}$ is \emph{categorically connected}, which means that every monomorphism $Y'\to Y$ is either an isomorphism or $Y'$ is an initial object. 

To state the Galois correspondence for (tame) infinite Galois categories, we need the notion of a Noohi group (see \cite[Definition 7.1.1 and Proposition 7.1.5]{BhattScholze}): a Hausdorff topological group which has a neighborhood basis of $1$ given by open (not necessarily normal) subgroups, and which is  Ra\^{i}kov complete. For a Noohi group $G$, we denote the category of discrete sets with continuous $G$ action by $G\text{-}\cat{Set}$. 

\begin{prop}[{\cite[Example 7.2.2 and Theorem 7.2.5]{BhattScholze}}] \label{igc prop} 
    Let $(\mc{C},F)$ be an infinite Galois category and $G$ a Noohi group. Then, the following statements are true.
    \begin{enumerate}[(a)]
        \item \label{icg prop 1}The group $\pi_1(\mc{C},F)$ with its compact-open topology is a Noohi group.
        \item The pair $(G\text{-}\cat{Set},F_G)$, where $F_G\colon G\text{-}\cat{Set}\to\cat{Set}$ is the forgetful functor, is a tame infinite Galois category with a canonical isomorphism $G\simeq \pi_1(G\text{-}\cat{Set},F_G)$.
        \item The natural map $\Hom((\mc{C},F),(G\text{-}\cat{Set},F_G))\to\Hom_{\cnts}(G,\pi_1(\mc{C},F))$ is a bijection.
        \item \label{icg prop 4} If $(\mc{C},F)$ is tame then $F$ induces an equivalence $F\colon \mc{C}\isomto \pi_1(\mc{C},F)\text{-}\cat{Set}$.
    \end{enumerate}
\end{prop}

\subsection{Geometric coverings form a tame infinite Galois category}
\label{ss:geom-cov-infty-gal}

In this section we arrive at our main result: geometric coverings form a tame infinite Galois category.

\begin{thm}\label{main tameness result}
    Let $X$ be a connected rigid $K$-space and let $\ov{x}$ be a geometric point of $X$. Then, the pair $(\Cov_X,F_{\ov{x}})$ is a tame infinite Galois category.
\end{thm}

From Proposition \ref{igc prop} we then obtain the following definition and corollary.

\begin{definition}
    Let $X$ be a connected rigid $K$-space and $\ov{x}$ a geometric point of $X$. Then, the Noohi group $\pi_1(\Cov_X,F_{\ov{x}})$ is called the \emph{geometric arc fundamental group} of the pair $(X,\ov{x})$ and is denoted $\pi_1^{\rm ga}(X,\ov{x})$.
\end{definition}

\begin{cor} 
    Let $X$ be a connected rigid $K$-space and let $\ov{x}$ be a geometric point of $X$. Then, the functor 
    \begin{equation*}
        F_{\ov{x}}\colon \Cov_X\to \pi_1^{\rm ga}(X,\ov{x})\text{-}\cat{Set}
    \end{equation*}
    is an equivalence of categories.
\end{cor}

We now start moving towards the proof of Theorem \ref{main tameness result}. The keystone intermediary result is the following statement about the existence of `paths' between geometric points.

\begin{thm} \label{pathsexist} 
    Let $X$ be a connected rigid $K$-space and let $\ov{x}$ and $\ov{y}$ be two geometric points of $X$. Then, there exists an isomorphism $F_{\ov{x}}\simeq F_{\ov{y}}$ of functors $\Cov_X\to \Set$.
\end{thm}

\begin{proof} 
Let us first assume that $X$ is a smooth and separated rigid $K$-curve. Let $\ov{x}$ and $\ov{y}$ be anchored at $x$ and $y$ respectively. If $\ov{x}^\mx$ is the associated maximal geometric point associated to $x$, then $F_{\ov{x}}\simeq F_{\ov{x}^\mx}$ via the definition of these fiber functors and the valuative criterion, and similarly for $\ov{y}$ and $\ov{y}^\mx$. Thus, we may assume that $\ov{x}$ and $\ov{y}$ are maximal geometric points. Let $\gamma$ be an arc in $X$ with left endpoint $\sep_X(x)$ and right endpoint $\sep_Y(y)$. Such an arc exists by Theorem \ref{thm:curves are arc connected}. By Proposition \ref{geometric intervals exist} we may upgrade this to a geometric arc $\ov{\gamma}$. Up to replacing $\ov{\gamma}$ by an equivalent arc we may assume that the left geometric endpoint of $\ov{\gamma}$ is $\ov{x}$ and the right geometric endpoint is $\ov{y}$. 

Define an isomorphism $\eta_{\ov{\gamma}}\colon F_{\ov{x}}\to F_{\ov{y}}$ as follows. For each object $Y$ of $\Cov_X$ we define $\eta_{\ov{\gamma}}(Y)\colon F_{\ov{x}}(Y)\to F_{\ov{y}}(Y)$ to be the map which associates to a geometric point $\ov{x}'$ in $F_{\ov{x}}(Y)$ the right geometric endpoint of the unique geometric arc $\ov{\gamma}'$ lifting $\ov{\gamma}$ with left geometric endpoint $\ov{x}'$. It is clear from the definition of uniqueness of geometric arc lifts, applied to both $\ov{\gamma}$ and $\ov{\gamma}^{\rm op}$, where $\ov\gamma^{\rm op}$ is $\ov\gamma$ with the opposite orientation, that this map is bijective.
It remains to show that this bijection is functorial in $Y$, which is straighforward to check.

Suppose now that $X$ is an arbitrary connected rigid $K$-space. Again, we may assume that the anchor points $x$ and $y$ are maximal. By Corollary~\ref{cor: general path connected} we can find a sequence of connected test curves in $X$ connecting $x$ and $y$. By iterating the procedure it suffices to assume that $x$ and $y$ are both contained in the image of a connected test curve $C\to X$. Let $\ov{x}'$ and $\ov{y}'$ be any lifts of $\ov{x}$ and $\ov{y}$ to $C$. Note then that there is an obvious map
\begin{equation*}
    \Isom_{\Cov_C}(F_{\ov{x'}},F_{\ov{y}'})\to \Isom_{\Cov_X}(F_{\ov{x}},F_{\ov{y}})
\end{equation*}
obtained via the natural identifications $F_{\ov{x}}(Y)=F_{\ov{x'}}(Y_C)$ and similarly for $\ov{y}$ and $\ov{y}'$. By the previous case we know that $\Isom_{\Cov_C}(F_{\ov{x'}},F_{\ov{y}'})\neq\emptyset$, and the conclusion follows.
\end{proof}

We now begin to show that $(\Cov_X,F_{\ov{x}})$ is a tame infinite Galois category. We begin by verifying Axioms \textbf{(IGC2)} and \textbf{(IGC3)}. Useful in this is the following result which shows that categorical connectedness in $\Cov_X$ coincides with topological connectedness.

\begin{prop} \label{connected equivalence} 
    Let $X$ be a connected rigid $K$-space and let $Y$ be an object of $\Cov_X$. Then, $Y$ is categorically connected if and only if $Y$ is topologically connected. Therefore, $Y$ admits a~unique decomposition $Y=\coprod_{i\in I}Y_i$ with each $Y_i$ a categorically connected object.
\end{prop}

\begin{proof}
The second statement clearly follows from the first, so it suffices to show this first claim. Suppose first that $Y$ is categorically connected. Write the connected component decomposition of $Y$ as $Y=\coprod_{i\in I} Y_i$. By Proposition \ref{disjoint union geometric coverings} we know that each $Y_i\to X$ is a geometric covering. One then clearly sees from the assumption that $Y$ is categorically connected that $I$ must be a singleton, and so $Y$ is connected.

Conversely, suppose that $Y$ is connected. Let $Y'\to Y$ be a monomorphism in $\Cov_X$ where $Y'$ is non-empty (and thus non-initial). Since $\Cov_X$ has fibered products by Proposition \ref{fibered product of geometric coverings}, $Y'\to Y$ being a monomorphism implies that the canonical map $Y'\to Y'\times_Y Y'$ is an isomorphism (see \stacks{01L3}). Again by Proposition \ref{fibered product of geometric coverings} the fibered product in $\Cov_X$ agrees with that in the category of rigid $K$-spaces, so $Y'\to Y$ is a monomorphism in rigid $K$-spaces. By \cite[Theorem~6.21]{FujiwaraKatoSurvey}, $Y'\to Y$ is an open embedding. Since $Y'\to Y$ is a geometric covering by Proposition~\ref{composition and base change geometric covering}~\ref{geom covering cancellation} we deduce from Proposition \ref{image of geometric covering clopen} that $Y'\to Y$ is surjective, and hence an isomorphism.
\end{proof}

\begin{cor}
    Let $X$ be a connected rigid $K$-space and $\ov{x}$ a geometric point of $X$. Then, the pair $(\mc{C},F_{\ov{x}})$ satisfies Axiom \emph{\textbf{(IGC2)}} and Axiom \emph{\textbf{(IGC3)}}.
\end{cor}

\begin{proof} 
Proposition \ref{connected equivalence} immediately implies Axiom \textbf{(IGC2)}. Axiom \textbf{(IGC3)} is elementary and follows from Lemma \ref{lem:card-bound}. We leave the details to the reader.
\end{proof}

We now move on to show that $(\Cov_X,F_{\ov{x}})$ satisfies Axiom \textbf{(IGC1)}.

\begin{prop} \label{cocomplete and finitely complete} 
    Let $X$ be a connected rigid $K$-space and let $\ov{x}$ be a geometric point. Then, the category $\Cov_X$ is cocomplete and finitely complete. Consequently, the pair $(\Cov_X,F_{\ov{x}})$ satisfies Axiom \emph{\textbf{(IGC1)}}.
\end{prop}

\begin{proof} 
We have already verified in Proposition \ref{disjoint union geometric coverings} and Proposition \ref{fibered product of geometric coverings} that $\Cov_X$ has fibered products and arbitrary coproducts. Thus, since $\Cov_X$ has a final object, it suffices to show that $\Cov_X$ has all coequalizers (e.g.\@ see \cite[\S V.2]{MacLane}).

To do this, let $W\rightrightarrows U$ be a pair of morphisms in $\Cov_X$. Let us consider the induced map $W\to U\times_X U$. By Proposition \ref{fibered product of geometric coverings} the natural morphism $U\times_X U\to X$ is a geometric covering, and since $W\to U\times_X U$ is a morphism over $X$ we deduce from Proposition~\ref{composition and base change geometric covering}~\ref{geom covering cancellation} that $W\to U\times_X U$ is a geometric covering. In particular, the image $R_0$ of $W\to U\times_X U$ is clopen by Proposition \ref{image of geometric covering clopen}.

Let $R$ be smallest equivalence relation containing $R$ in $U \times_X U$. It can be obtained in the following way. Define $R_1$ to be the symmetrization of $R_0$, i.e.\ $R_1 = R_0 \cup R_0^{\mathrm{inv}}$, where $R_0^{\mathrm{inv}}$ is the image of $R_0$ via the automorphism of $U \times_X U$ that switches the factors. Next, consider $R_1^n$ defined as the image of the map
\begin{displaymath}
    R_1 \times_{p,U,q} R_1 \times_{p,U,q} R_1 \times_{p,U,q} \ldots \times_{p,U,q} R_1 \to (U \times_X U)\times_{p,U,q}\ldots \times_{p,U,q} (U \times_X U)  \stackrel{\sim}{\to} U\times_X U
\end{displaymath}
Here the left hand side is the $n$-fold fibered product and $p$ (resp.\@ $q$) is the left (resp.\@ right) projection $U\times_X U\to U$. Similarly to the argument given in the second paragraph, each $R_1^n$ is a clopen subset of $U\times_X U$. Define $R$ to be the union $\bigcup_n R_1^n$. Note that since $U$ is a locally connected topological space, this union will be clopen. Now, the coequalizer of $W \rightrightarrows U$, if it exists, will be isomorphic to the coequalizer of $R \rightrightarrows U\times_X U$, by construction. 

Let $\mc{F}$ be the sheaf on the \'etale site of $X$ given by the sheaf quotient $(U\times_X U)/R$. As $R \to U\times_X U$ is a closed immersion, we can apply \cite[Theorem 3.1.5]{Warner} to get that that $\mc{F}$ is representable. Let us write the representing object by $Q$. By loc.\@ cit.\@ we also have that $Q\to X$ is separated and $U\to Q$ is surjective and \'etale. We claim that $Q\to X$ is an object of $\Cov_X$. The fact that $Q\to X$ is \'etale is simple since $U\to X$ is \'etale and $U\to Q$ is \'etale and surjective (see Lemma \ref{etale image lemma})
. So then, one sees that $Q \to X$ is \'etale and surjective. Since one has a surjection $U\to Q$ over $X$, we deduce that $Q\to X$ is a geometric covering by Proposition~\ref{image of geometric covering is geometric covering}.
\end{proof}

Finally, to show that $(\Cov_X,F_{\ov{x}})$ is an infinite Galois category we need only to show that Axiom \textbf{(IGC4)} holds. This is done in the following. 

\begin{prop}
    Let $X$ be a connected rigid $K$-space and let $\ov{x}$ be a geometric point of $X$. Then, the pair $(\Cov_X,F_{\ov{x}})$ satisfies Axiom \emph{\textbf{(IGC4)}}.
\end{prop}

\begin{proof} 
Let us first show that $F_{\ov{x}}$ is faithful. Let $Y_1\to Y_2$ be a morphism in $\Cov_X$. Note then that the graph morphism $Y_1\to Y_1\times_X Y_2$ is a geometric covering by combining Proposition~\ref{composition and base change geometric covering}~\ref{geom covering cancellation} and Proposition \ref{fibered product of geometric coverings}. In particular, it is an \'etale monomorphism, and so an open embedding by \cite[Theorem~6.21]{FujiwaraKatoSurvey}. By Proposition \ref{image of geometric covering clopen} it has clopen image, and so $Y_1\to Y_1\times_X Y_2$ is an isomorphism onto some connected component of $Y_1\times_X Y_2$. So, suppose that $f,g\colon Y_1\to Y_2$ are morphisms in $\Cov_X$ such that $F_{\ov{x}}(f)=F_{\ov{x}}(g)$ for some geometric point $\ov{x}$ of $X$. To show that $f=g$ it suffices to show that the connected components of $Y_1\times_X Y_2$ corresponding to $f$ and $g$ agree. By Theorem \ref{pathsexist} we have that $F_{\ov{x}}(f)=F_{\ov{x}}(g)$ for all geometric points $\ov{x}$ of $X$. If $\ov{y}$ is a geometric point of $Y_1$ lying over the geometric point $\ov{x}$ of $X$, then the image of this geometric point under the graph map $Y_1\to Y_1\times_X Y_2$ for $f$ is $(\ov{y},F_{\ov{x}}(f)(\ov{y}))$, and similarly for $g$. In particular, since $F_{\ov{x}}(f)=F_{\ov{x}}(g)$ for all geometric points $\ov{x}$ of $X$, we see that the maps $Y_1\to Y_1\times_X Y_2$ induced by $f$ and $g$ have the same images as desired.

To see that $F_{\ov{x}}$ is conservative, suppose that $f\colon Y_1\to Y_2$ is a morphism in $\Cov_X$ such that $F_{\ov{x}}(f)$ is a bijection for some geometric point $\ov{x}$. By Proposition~\ref{composition and base change geometric covering}~\ref{geom covering cancellation},  $f$ is \'etale. Moreover, by Theorem \ref{pathsexist} see that $F_{\ov{x}}(f)$ is a bijection for all geometric points $\ov{x}$ of $X$. This clearly implies that one has the unique lifting property of geometric points as in \cite[Theorem~6.21]{FujiwaraKatoSurvey}. By \cite[Theorem~6.21]{FujiwaraKatoSurvey}, $f$ is an isomorphism as desired.

To see that $F_{\ov{x}}$ is cocontinuous and finitely continuous, we note that since $\Cov_X$ has a~final object it suffices to show that $F_{\ov{x}}$ commutes with arbitrary coproducts, coequalizers, and fibered products (see \cite[\S V.2]{MacLane}). It clearly commutes with arbitrary coproducts and fibered products, and thus it suffices to show it commutes with coequalizers. Let $W\rightrightarrows U$ be a pair of arrows in $\Cov_X$. Recall that in the proof of Proposition \ref{cocomplete and finitely complete} one identified 
\begin{equation*}
    \mathrm{Coeq}\left(W\rightrightarrows U\right)\simeq (U\times_X U)/R
\end{equation*}
with notation as in that proposition. Since quotients of \'etale equivalence relations commute with pullback (cf.\@ \stacks{03I4}), for each morphism $V\to X$ of adic spaces, one has that 
\begin{equation*}
    \mathrm{Coeq}\left(W\rightrightarrows U\right)_V\simeq ((U\times_X U)/R)_V\simeq (U_V\times_V U_V)/R_V\simeq \mathrm{Coeq}\left(W_V\rightrightarrows U_V\right)
\end{equation*}
Applying this to the case when $V=\ov{x}$ shows that $F_{\ov{x}}$ commutes with coequalizers.
\end{proof}

Thus, the only part of Theorem \ref{main tameness result} left to be proven is the claim that the infinite Galois category $(\Cov_X,F_{\ov{x}})$ is tame. 

\begin{prop} \label{tameness follows from pathsexist} 
    Let $X$ be a connected rigid $K$-space, $\ov{x}$ a geometric point of $X$, and $Y$ a~categorically connected object of $\Cov_X$. Then, $\Aut(F_{\ov{x}})$ acts transitively on $F_{\ov{x}}(Y)$. 
\end{prop}

\begin{proof} 
As in the beginning of the proof of Proposition \ref{pathsexist}, we may assume that $\ov{x}$ is a~maximal geometric point. Let $\ov{y}_1$ and $\ov{y}_2$ be elements of $F_{\ov{x}}(Y)$. Since $Y$ is connected we obtain from Theorem \ref{pathsexist} an \'etale path $\eta\colon F_{\ov{y}_1}\isomto F_{\ov{y}_2}$ of functors on $\Cov_Y$. Consider the base change functor $b\colon \Cov_X\to \Cov_Y$. One easily sees that $F_{\ov{y}_i}\circ b\simeq F_{\ov{x}}$. Thus, from the isomorphism $\eta$ we obtain an isomorphism $\eta'$ given as the composition
\begin{equation*}
    F_{\ov{x}}\simeq F_{\ov{y}_1}\circ b\xrightarrow{\eta}F_{\ov{y}_2}\circ b\simeq F_{\ov{x}}
 \end{equation*}
By functoriality of $\eta$ and the definition of $\eta'$ one has a commutative diagram
 \begin{equation*}
     \xymatrixcolsep{5pc}\xymatrixrowsep{2pc}
     \xymatrix{
        \{\ov{y}_1\}=F_{\ov{y}_1}(Y)\ar[r]^{\eta(Y)}\ar[d]_{\Delta_{Y/X}} & F_{\ov{y}_2}(Y)=\{\ov{y}_2\}\ar[d]^{\Delta_{Y/X}}\\
        F_{\ov{y}_1}(Y\times_X Y)\ar[r]^-{\eta(Y\times_X Y)}\ar@{=}[d] & F_{\ov{y}_2}(Y\times_X Y)\ar@{=}[d]\\
        F_{\ov{x}}(Y)\ar[r]^{\eta'(Y)} & F_{\ov{x}}(Y)
    }
 \end{equation*}
 where $\Delta_{Y/X}\colon  Y\to Y\times_X Y$ is the diagonal map. This clearly shows that $\eta(Y)$ takes $\ov{y}_1$ to $\ov{y}_2$. Since $\ov{y}_1$ and $\ov{y}_2$ were arbitrary, the conclusion follows. 
\end{proof}

\begin{rem} \label{rem:tame-paths}
Why do we need to pullback to curves to define geometric coverings? And why do we use arcs and not paths? Roughly, the reason is that arcs in higher-dimensional Berkovich spaces can be very wiggly, which may cause trouble at various steps of the proofs. For example, Lemma~\ref{lem:finite-pi0} is false in higher dimension. In fact, we expect there to be a more general notion of a tame path in a Berkovich analytic space, such that every two points in the same connected component are the endpoints of a tame path. This would likely allow one to develop a cleaner definition of a geometric covering, without the crutch of base changing to curves.
\end{rem}

\appendix 

\section{Curve-connectedness of rigid varieties}
\label{curve-connectedness appendix}

\begin{definitionA} \label{def:curve-conn}
	A rigid $K$-space is \emph{curve-connected} if for every two classical points $x_0, x_1$ of $X$ there exists a sequence of morphisms $C_i\to X$ ($i=0, \ldots, n$) over $K$ where each $C_i$ is a~connected affinoid rigid $K$-curve and we have $x_0\in \im(C_0\to X)$, $x_1\in \im(C_n\to X)$, and $\im(C_i\to X)\cap \im(C_{i-1}\to X)\neq\emptyset$ for $i=1, \ldots, n$.
\end{definitionA}

By \cite[Theorem~3.3.6]{ConradComponents}, one can assume that the curves $C_i$ are smooth over $K$.

\begin{propA} \label{curve-connectedness prop}
	Let $X$ be a connected rigid $K$-space. Then, $X$ is curve-connected.
\end{propA}

\noindent
The analogous result was proven by de~Jong in \cite[Theorem~6.1.1]{deJongCrystal} for quasi-compact rigid $K$-spaces in the case when $K$ is discretely valued. Our proof follows his, but requires some non-trivial alterations due to the fact that $\cO_K$ is non-Noetherian. It is also worth pointing out that Berkovich obtained similar results for Berkovich spaces (see \cite[Theorem 4.1.1]{BerkovichIntegration}) which, in the language of adic spaces, are partially proper over $\Spa(K)$.

\begin{proof}
Idea of the proof: If $X$ is a connected affine variety, one can use the Bertini theorem to find a connected hypersurface passing through two given points. In the situation at hand, if $X=\Spa(A)$ is affinoid, we can construct a suitable connected hyperplane section of its reduction $\widetilde X = \Spec(\widetilde A)$, and lift it to the formal model $\Spf(A^\circ)$. The main difficulty, resolved carefully below, is to ensure that the resulting hypersurface remains connected on the generic fiber. 
\beginsteps

\step[\textit{We may assume that $X$ is affinoid}]
Let $U$ be the union of all connected quasi-compact opens containing $x_0$. Then $U$ is (open and) closed: if $y\notin U$ and $V$ is a connected affinoid neighborhood of $y$, then $V\cap U=\emptyset$, otherwise $V\cup U$ is quasi-compact and connected. Since $X$ is connected, $x_1\in U$, i.e.\ $x_0$ and $x_1$ admit a connected quasi-compact neighborhood $W$. Write $W=\bigcup_{i=1}^n X_i$ with $X_i$ affinoid, which we may assume are connected and $X_i\cap X_{i-1}\neq\emptyset$ (in particular, the intersection has classical points), and it is enough to show the result for the $X_i$. 

\step[\textit{We  may assume that $X$ is geometrically connected and geometrically normal}]
In the proof, we can freely pass to a finite extension $L$ of $K$, because every connected component of $X_L$ surjects onto $X$. By \cite[\S3.2]{ConradComponents}, we may assume that $X$ is geometrically connected. By \cite{ConradComponents}, affinoid $K$-algebras are excellent and we have the normalization map $X^{\rm n
}\to X$, a finite surjective morphism. By \cite[Theorem 3.3.6]{ConradComponents}, after passing to a finite extension of $K$ we may therefore assume that $X$ is geometrically normal.

\medskip

\step[\textit{Reduced Fiber Theorem}]
Write $X=\Spa(A)$. By the Reduced Fiber Theorem \cite{GrauertRemmert,BoschLutkebohmertRaynaud}, passing to a~finite extension of $K$ we may assume that $A^\circ$ is a topologically finitely presented $\cO_K$-algebra such that $A^\circ\otimes k$ is geometrically reduced and the formation of $A^\circ$ is compatible with further finite extensions of~$K$.

\medskip

Write $\widetilde{A} = A^\circ\otimes k$ and $\widetilde{X} = \Spec(\widetilde{A})$, and set $d=\dim X$. We denote the irreducible components of $\widetilde{X}$ by $Z_1, \ldots, Z_r$. Note that we have $d=\dim Z_i$ for all $i$. To see this, let $\widetilde{U}_i\subseteq \widetilde{X}$ be an affine open subset contained in $Z_i$, so that $\dim \widetilde{U}_i=\dim Z_i$, and let $U_i\subseteq \Spf(A^\circ)$ be the corresponding open formal subscheme. Then the generic fiber $U_{i,\eta}$ is an affinoid subdomain of $X$ and hence has dimension $d$. We conclude by \cite[Theorem A.2.1]{ConradModular}.

\step[\textit{Noether normalization}]
Since $\widetilde{X}$ is geometrically reduced, Lemma~\ref{lem:sep-nn} below yields a finite and generically \'etale morphism $\tilde{f}\colon \widetilde{X}\to \mathbf{A}^d_k$ (if $k$ is finite, we might need to pass to a~finite extension of $K$). Choosing lifts of $\tilde{f}^*(x_i)$ to $A^\circ$, we lift $\tilde{f}$ to a map 
\[
    f\colon \Spec(A^\circ) \to \Spec(\cO_K\langle x_1, \ldots, x_d\rangle), 
\]
which we claim is finite, and \'etale at the points of $\widetilde{X}$ where $\tilde f$ is \'etale. Then the morphism $f_\eta \colon X\to \mb{D}^d_K$ induced by $K\langle x_1, \ldots, x_d\rangle \to A$ is finite as well.

To prove this claim, set $B^\circ = \cO\langle x_1, \ldots, x_d\rangle$ and let $f_0\colon B^\circ/\varpi B^\circ \to A^\circ/\varpi A^\circ$ be the induced map. Since $f_0$ is a thickening of $\widetilde{f}$ in the sense of \stacks{04EX} and $A^\circ/\varpi A^\circ$ is of finite presentation over $\cO_K/\varpi\cO_K$,  by \stacks{0BPG}, the finiteness of $\widetilde{f}$ implies the finiteness of $f_0$. By \cite[Chapter I, Proposition 4.2.3]{FujiwaraKato}, this implies that $f$ is finite, and therefore also finitely presented \cite[7.3/4]{BoschLectures}.

Let now $x\in \widetilde{X}\subseteq \Spec(A^\circ)$ be a point where $\tilde f$ is \'etale; we will show that $f$ is \'etale at $x$. Set $y=f(x)$. As $f$ is finitely presented and $\Spec(A^\circ)_y\to\Spec(k(y))$ is \'etale at $x$ (as it agrees with $\widetilde{f}$) it is enough by \stacks{01V9} to show that $\cO_{\Spec(B^\circ),y}\to \cO_{\Spec(A^\circ),x}$ is flat. Note that 
\begin{equation*}
    \left(\cO_{\Spec(A^\circ), x}\to \cO_{\Spf(A^\circ),x}\right)=\varinjlim_{x\in D(f)} \left(A^\circ[f^{-1}]\to A^\circ\langle f^{-1}\rangle\right)
\end{equation*} 
and the right-hand side is flat by Gabber's lemma \cite[8.2/2]{BoschLectures}, but the map is also local by \cite[Remark 7.2/1]{BoschLectures}. Thus, $\cO_{\Spec(A^\circ), x}\to \cO_{\Spf(A^\circ),x}$ is faithfully flat, and similarly $\cO_{\Spec(B^\circ), y} \to \cO_{\Spf(B^\circ), y}$ is faithfully flat. Finally, the map $\cO_{\Spf(B^\circ), y}\to \cO_{\Spf(A^\circ),x}$ is flat by \cite[Proposition~I,  5.3.11]{FujiwaraKato} and \stacks{0CF4}. Note that the compositions
\begin{equation*}
    \mc{O}_{\Spec(B^\circ),y}\to \mc{O}_{\Spec(A^\circ),x}\to \mc{O}_{\Spf(A^\circ),x},\qquad \mc{O}_{\Spec(B^\circ),y}\to\mc{O}_{\Spf(B^\circ),y}\to\mc{O}_{\Spf(A^\circ),x}
\end{equation*}
are equal. By the discussion above the second composition is flat, and thus the first. As $\mc{O}_{\Spec(A^\circ),x}\to \mc{O}_{\Spf(A^\circ),x}$ is faithfully flat we see $  \mc{O}_{\Spec(B^\circ),y}\to \mc{O}_{\Spec(A^\circ),x}$ is flat as desired.

\begin{lemA}[{Generically \'etale Noether normalization}] \label{lem:sep-nn}
    Let $Y$ be an affine scheme of finite type over a field $k$ which is geometrically reduced and whose irreducible components have the same dimension $d$. Then there exists a finite and generically \'etale map $Y\to \mathbf{A}^d_k$.
\end{lemA}

\begin{proof}
Apply \cite[Lemma~6]{KedlayaMoreEtale}\footnote{As pointed out by Ofer Gabber, \cite[Lemmas~5 and 6]{KedlayaMoreEtale} need an extra assumption that $\dim(D)<n$.} to the closure of $Y\subseteq \mb{A}^n_k$ in $\mb{P}^n_k$ and a zero-dimensional subscheme of the smooth locus of $Y$ meeting every irreducible component.
\end{proof}

\step
We claim that it is enough to show:
\begin{quote}
    \emph{If $d\geq 2$, there is a hyperplane $H\subseteq \mb{D}^d_K$ such that $f_\eta^{-1}(H)\subseteq X$ is connected.} 
\end{quote}
Here, by a hyperplane we mean the zero set $H=V(h)$ of a linear form $h = \sum_{i=1}^d a_i x_i - a$ with $a_i,a\in\cO_K$ with $a_i$ not all zero in $k$.

The argument for this is exactly as in \cite{deJongCrystal}: we prove the entire theorem by induction on $d=\dim X\geq 0$. Since the assertion is clear for $d\leq 1$, we assume $d\geq 2$. For the induction step, we find $H$ as in the claim, and since $f_\eta^{-1}(H)$ is curve-connected by the induction assumption, it suffices to connect a given classical point $x\in X$ to a point $y\in f_\eta^{-1}(H)$ by a curve in $X$. Let $p\colon \mb{D}^d_K\to H$ be a linear projection (so $p|_H = \op{id}_H$); then $C=p^{-1}(p(f(x)))$ is a connected curve (in fact, isomorphic to $\mb{D}^1_{L}$ for a finite extension $L/K$) connecting $f(x)$ with $H$. Let $C'$ be the connected component of $f^{-1}(C)$ containing $x$. We claim that $\dim C'=1$; otherwise, $x$ is an isolated point of $f^{-1}(C)$, which is impossible because the map $\Spec(A)\to \Spec(K\langle x_1, \ldots, x_d\rangle)$ is open by \stacks{0F32}\footnote{Indeed, if $I\subseteq B=K\langle x_1, \ldots, x_d\rangle$ is the ideal of $C$, then $x$ is an isolated point of $f^{-1}(C)=\Spa(A/IA)$ if and only if it is an isolated point of $\Spec(A/IA)$, as both mean that the local ring $(A/IA)_x$ is Artinian. If $U\subseteq \Spec(A)$ is an open subset with $U\cap \Spec(A/IA) = \{x\}$, then its image in $\Spec(B)$ is an open neighborhood of $f(x)$, and hence it contains the generic point of $\Spec(B/I)$, contradiction.}.  The map $C'\to C$ is finite, and hence surjective, and therefore $C'$ connects $x$ with a point in $f^{-1}(H)$.

\medskip

In the following steps, we shall construct a hyperplane $\widetilde H\subseteq \mb{A}^d_k$ and show that every hyperplane $H\subseteq \mb{D}^d_K$ lifting $\widetilde H$ has connected preimage in $X$. 

\step[\textit{Construction of $H$}]
For a~decomposition $I\cup J=\{1, \ldots, r\}$ with $I$, $J$ non-empty and disjoint, we set 
\[
    Z_{I,J} = (\bigcup_{i\in I}Z_i)\cap (\bigcup_{j\in J}Z_j),
\]
treated as a reduced subscheme of $\widetilde{X}$. We claim as in \cite[\S 6.4]{deJongCrystal} that $\dim Z_{I,J}=d-1$ for every such $I$, $J$. To this end, note that $\widetilde{X}\setminus Z_{I,J}$ is disconnected, and hence so is $\mf{U}_\eta \subseteq X$ where $\mf{U}\subseteq \Spf(A^\circ)$ is the open formal subscheme supported on $\widetilde{X}\setminus Z_{I,J}$. But, if $\dim Z_{I,J}<d-1$, then by \cite[Satz 2]{Lut2} we have $\cO(\mf{U}_\eta)= A$, which does not have non-trivial idempotents. 

We denote by $\widetilde{T}\subseteq \widetilde{X}$ the closed subscheme where $\tilde f$ is not \'etale, and by $\widetilde{S}\subseteq \widetilde{X}$ the closed subscheme where $\tilde f$ is not flat. By construction, $\dim\widetilde{T}<d$. By Miracle Flatness \stacks{00R4}, $\widetilde{S}$ is the non-Cohen--Macaulay locus of $\widetilde{X}$, see \stacks{00RE}. Since $\widetilde{X}$ is reduced, it is $S_1$ \stacks{031R}, and therefore we have $\dim \widetilde{S}<d-1$.

By the Bertini theorem \cite[Th\'eor\`eme 6.3]{Jouanolou}, a generic hyperplane (which exists after replacing $k$ by a finite extension if $k$ is finite) $\widetilde H = V(\tilde h) \subseteq \mb{A}^d_k$, $\tilde h=\sum_{i=1}^d \tilde a_i x_i - \tilde a$ ($\tilde a_i, \tilde a\in k$) satisfies the following properties:
\begin{enumerate}[(1)]
    \item The intersections $\tilde f^{-1}(\widetilde H)\cap Z_i$ ($i=1, \ldots, r$) are irreducible of dimension $d-1$ and generically \'etale over $\widetilde H$.
    \item For every decomposition $I\cup J=\{1, \ldots, r\}$ with $I$ and $J$ non-empty and disjoint, $\tilde f^{-1}(\widetilde{H}) \cap Z_{I,J}$ has a component of dimension $d-2$, and all such components are generically flat over $\widetilde H$.  
\end{enumerate} 

Indeed, \cite[Th\'eor\`eme 6.3(1b,4)]{Jouanolou} shows that an open set of hyperplanes will have the property that $\tilde f^{-1}(\widetilde H)\cap Z_i$ are irreducible of dimension $d-1$. To ensure \'etaleness in (1) it suffices to choose $\widetilde H$ not contained in $\tilde f(\widetilde{T})$. If $V\subseteq Z_{I,J}$ is an irreducible component of dimension $d-1$, then by \cite[Th\'eor\`eme 6.3(1b)]{Jouanolou} for a generic $\widetilde H$, the preimage $\tilde f^{-1}(\widetilde H)\cap V$ is of dimension $d-2$. To ensure flatness in (2), we choose a hyperplane whose intersection with $\tilde f(V)$ is not contained in $\tilde f(\tilde S)$. 

We let $h=\sum_{i=1}^d a_i x_i - a$ ($a_i, a\in \cO_K$) be any lifting of $h$, and set $H=V(h)\subseteq \mb{D}^d_K$. In the remaining two steps, we shall prove that $f^{-1}_\eta(H)\subseteq X$ is connected. By the claim in Step~5, this will finish the proof. To this end, if $f_\eta^{-1}(H)$ is disconnected, then we have $\Spec(A^\circ/h A^\circ) = T_1\cup T_2$ for two non-empty closed subsets $T_1$, $T_2$ such that $T_1\cap T_2$ is non-empty and set-theoretically contained in $\widetilde X$. In the final two steps, we shall derive a contradiction. 

\step[\textit{$T_1\cap T_2$ does not contain any $\tilde f^{-1}(\widetilde H) \cap Z_i$}] Let $\xi_i$ be the generic point of $\tilde f^{-1}(\widetilde{H})\cap Z_i$; we will show that $\xi_i\notin T_1\cap T_2$. It is enough to show that $A^\circ_{\xi_i}/hA^\circ_{\xi_i}$ is a domain. Indeed, since $\dim(T_1)=\dim(T_2)=d$ if $\xi_i\in T_1\cap T_2$ then $\Spec(A^\circ_{\xi_i}/hA^\circ_{\xi_i})$ would contain multiple components of $\Spec(A^\circ/hA^\circ)$, but this would contradict that $A^\circ_{\xi_i}/hA^\circ_{\xi_i}$ is a domain. 
To see that $A^\circ_{\xi_i}/hA^\circ_{\xi_i}$ is a domain, consider the localization $R$ of $\cO_K\langle x_1, \ldots, x_d\rangle$ at the generic point of $\widetilde H$ and the map $R\to A^\circ_{\xi_i}$ induced by $f$. Since $f$ is \'etale in a neighborhood of $\xi_i$, the induced map $R/hR \to A^\circ_{\xi_i}/hA^\circ_{\xi_i}$ is the composition of an \'etale morphism and a~localization; in particular, it is weakly \'etale.  Since $\cO_K\langle x_1, \ldots, x_d\rangle/h \simeq \cO_K\langle y_1, \ldots, y_{d-1}\rangle$ is normal, so is $R/hR$. Thus, by \stacks{0950}, we conclude that $A^\circ_{\xi_i}/h A^\circ_{\xi_i}$ is normal as well, in particular it is a domain.

\step[\textit{The contradiction}] It follows from the previous step that if we set
\begin{equation*}I=\{i\,:\, \xi_i\in T_1\},\qquad J=\{j\,:\, \xi_j\in T_2\}
\end{equation*}
then these form a non-trivial partition of $\{1, \ldots, r\}$. Therefore $T_1\cap T_2\cap \widetilde X$ is contained in $Z_{I,J}$. Note though that since $T_1$ and $T_2$ both have codimension one in $\Spec(A^\circ)$ the intersection $T_1\cap T_2\cap \widetilde X$ necessarily contains a component $V$ of $\widetilde f^{-1}(\widetilde H)\cap Z_{I,J}$. This component can be taken to have codimension two in $\widetilde X$, and by our choice of hyperplane, we know that this component is generically Cohen--Macaulay and so not contained in $\widetilde{S}$.

Let $\zeta$ be the generic point of $V$. By the previous paragraph we know that $\zeta$ is not contained in $\widetilde{S}$ and therefore $\widetilde A_\zeta$ is a two-dimensional, Cohen--Macaulay, local ring. Moreover, $\widetilde{h}$ is a~nonzerodivisor of $\widetilde A_\zeta$ contained in its maximal ideal. Note that since $\widetilde{A}_\zeta/\widetilde{h} \widetilde{A}_\zeta$ is not zero-dimensional its maximal ideal contains some nonzerodivisor. Let $\widetilde{g}$ in $\widetilde A_\zeta$ be the lift of such an element. Then, $(\widetilde{g}, \widetilde{h})$ is a~regular sequence in $\widetilde{A}_\zeta$. Let $\varpi$ be the pseudouniformizer from above and choose a~lift $g$ of $\widetilde{g}$ in $A^\circ_\zeta$. Note then that $(\varpi,g,h)$ is a~regular sequence in $A^\circ_\zeta$. Indeed, this is clear since $A^\circ_\zeta/\varpi A^\circ_\zeta$ is a local ring and $(g,h)$ have images in this ring that, in the further quotient ring, $\widetilde{A}_\zeta$ have images forming a regular sequence. Therefore $(\varpi, g)$ is a regular sequence in the two-dimensional $A^\circ_\zeta/h A^\circ_\zeta$. The fact that we can permute a regular sequence follows from the argument given in \stacks{00LJ}. Indeed, while $A^\circ_\zeta/hA^\circ_\zeta$ is not Noetherian it is coherent (e.\@g.\@ by \cite[Chapter~0, Corollary~9.2.8]{FujiwaraKato}) and so the annihilator of any element of $A^\circ_\zeta/hA^\circ_\zeta$ is finitely generated, which is all the referenced argument requires.

By Lemma~\ref{local ring connected} below we know that $W = \Spec(A^\circ_\zeta/hA^\circ_\zeta)\setminus V(\varpi,g)$ is connected. Note though that since $T_1\cup T_2=\Spec(A^\circ/hA^\circ)$ by construction, that $T_1'\cup T_2'=W$ where $T_i'=T_i\cap W$. Each $T_i'$ is closed in $W$, and we have that $T_1'\cap T_2'\subseteq V(\varpi)\cap W$. However, we have $V(\varpi)\cap W)=\varnothing$: note that $V(\varpi)\cap W = \Spec(\widetilde{A}_\zeta/\widetilde{h}\widetilde{A}_\zeta) \setminus V(\widetilde{g})$ and this set is contained in the  union of $\xi_i$'s that are in the intersection $T_1 \cap T_2$. As shown above, the set of those $\xi_i$'s is empty. Thus, $T_1'$ and $T_2'$ form a disconnection of $W$. Contradiction.
\end{proof}

\begin{lemA}\label{local ring connected} 
    Let $R$ be a local ring, let $x, y\in R$ be such that $(x, y)$ and $(y, x)$ are regular sequences, and set $W=\Spec(R) \setminus V(x, y)$. Then $\Gamma(W, \cO_W) = R$; in particular, $W$ is connected.
\end{lemA}

\begin{proof}
Since $W=D(x)\cup D(y)$, we have $\Gamma(W, \cO_W)=\ker(R_x\times R_y\to R_{xy})$. As $x$ is a nonzerodivisor, $R\to R_x$ is injective, and hence so is $R\to \Gamma(W, \cO_W)$. To show it is surjective, let us take an element $(a/x^n, b/y^m)\in R_x\times R_y$ in the kernel, i.e.\ $(xy)^N(ay^m - bx^n)=0$. Since $xy$ is a~nonzerodivisor, we have $ay^m = bx^n$. We may assume that $n=0$ or that $a \notin xR$. If $n>0$, then we have $ay^m = 0$ in $R/xR$, and hence $a\in xR$ since $y$ is a~nonzerodivisor in $R/xR$. Therefore $n=0$, analogously $m=0$, and hence $a=b$.
\end{proof}

We can bootstrap this up to connect points on connected rigid $K$-spaces which are not necessarily classical, at the expense of base field extension.

\begin{corA} \label{cor: general path connected}
    Let $X$ be a connected rigid $K$-space. Fix maximal points $x$ and $y$ in $X$. Then, there exists a complete extension $L$ of $K$ and smooth connected affinoid $L$-curves $C_i$ with maps $C_i \to X$  such that for all $i$ we have that $\mathrm{im}(C_i\to X)\cap \mathrm{im}(C_{i+1}\to X)$ is non-empty, $x\in \mathrm{im}(C_1\to X)$, and $y\in \mathrm{im}(C_m\to X)$.
\end{corA}

\begin{proof} As in the second step of the proof of Proposition \ref{curve-connectedness prop} we may assume that $X$ is geometrically connected. Let us then note that by Gruson's theorem (\cite[Th\'{e}or\`{e}me~1]{Gruson}) one has that the completed tensor product $k(x)\widehat{\otimes}_K k(y)$ is non-zero. Thus, by taking a point of $\mc{M}(k(x)\widehat{\otimes}_K k(y))$ (which exists by \cite[Theorem 1.2.1]{BerkovichSpectral}) and looking at its residue field we get a valued extension $L$ of $K$ containing both $k(x)$ and $k(y)$. This gives two maps $\mc{M}(L) \rightrightarrows X$. They give rise to two $L$-points $x'$ (resp.\ $y'$) of $X_L$ mapping to $x$ (resp.\ $y$). We then note that by Proposition \ref{curve-connectedness prop} there exists smooth, connected, affinoid curves $C_1,\ldots,C_n$ over $L$ and morphisms $C_i\to X_L$ satisfying $x'\in\mathrm{im}(C_1\to X_L)$, $y'\in\im(C_n\to X_L)$, and $\im(C_i\to X_L)\cap \im(C_{i+1}\to X_L)$ is non-empty. Clearly then the compositions $C_i\to X_L\to X$ satisfy the desired properties.
\end{proof}

\bibliographystyle{amsalpha}

\begin{thebibliography}{Stacks}

\bibitem[ALY21]{ALY1P2}
Piotr Achinger, Marcin Lara, and Alex Youcis, \emph{Variants of the de {J}ong
  fundamental group}, arXiv preprint
  \href{http://arxiv.org/abs/2203.11750}{arXiv:2203.11750}, 2021.

\bibitem[And03]{AndreLectures}
Yves Andr\'{e}, \emph{Period mappings and differential equations. {F}rom
  {$\mathbf{C}$} to {$\mathbf{C}_p$}}, MSJ Memoirs, vol.~12, Mathematical
  Society of Japan, Tokyo, 2003, T\^{o}hoku-Hokkaid\^{o} lectures in arithmetic
  geometry, With appendices by F. Kato and N. Tsuzuki. \MR{1978691}

\bibitem[Ber90]{BerkovichSpectral}
Vladimir~G. Berkovich, \emph{Spectral theory and analytic geometry over
  non-{A}rchimedean fields}, Mathematical Surveys and Monographs, vol.~33,
  American Mathematical Society, Providence, RI, 1990. \MR{1070709}

\bibitem[Ber93]{BerkovichEtale}
\bysame, \emph{\'{E}tale cohomology for non-{A}rchimedean analytic spaces},
  Inst. Hautes \'{E}tudes Sci. Publ. Math. (1993), no.~78, 5--161 (1994).
  \MR{1259429}

\bibitem[Ber99]{BerkovichContractible}
\bysame, \emph{Smooth {$p$}-adic analytic spaces are locally contractible},
  Invent. Math. \textbf{137} (1999), no.~1, 1--84. \MR{1702143}

\bibitem[Ber07]{BerkovichIntegration}
\bysame, \emph{Integration of one-forms on {$p$}-adic analytic spaces}, Annals
  of Mathematics Studies, vol. 162, Princeton University Press, Princeton, NJ,
  2007. \MR{2263704}

\bibitem[BGR84]{BGR}
Siegfried Bosch, Ulrich G\"{u}ntzer, and Reinhold Remmert,
  \emph{Non-{A}rchimedean analysis}, Grundlehren der Mathematischen
  Wissenschaften, vol. 261, Springer-Verlag, Berlin, 1984. \MR{746961}

\bibitem[BLR95]{BoschLutkebohmertRaynaud}
Siegfried Bosch, Werner L\"{u}tkebohmert, and Michel Raynaud, \emph{Formal and
  rigid geometry. {IV}. {T}he reduced fibre theorem}, Invent. Math.
  \textbf{119} (1995), no.~2, 361--398. \MR{1312505}

\bibitem[Bos14]{BoschLectures}
Siegfried Bosch, \emph{Lectures on formal and rigid geometry}, Lecture Notes in
  Mathematics, vol. 2105, Springer, Cham, 2014. \MR{3309387}

\bibitem[Bra12]{Brazas}
Jeremy Brazas, \emph{Semicoverings: a generalization of covering space theory},
  Homology Homotopy Appl. \textbf{14} (2012), no.~1, 33--63. \MR{2954666}

\bibitem[BS15]{BhattScholze}
Bhargav Bhatt and Peter Scholze, \emph{The pro-\'{e}tale topology for schemes},
  Ast\'{e}risque (2015), no.~369, 99--201. \MR{3379634}

\bibitem[Con99]{ConradComponents}
Brian Conrad, \emph{Irreducible components of rigid spaces}, Ann. Inst. Fourier
  (Grenoble) \textbf{49} (1999), no.~2, 473--541. \MR{1697371}

\bibitem[Con06]{ConradModular}
\bysame, \emph{Modular curves and rigid-analytic spaces}, Pure Appl. Math. Q.
  \textbf{2} (2006), no.~1, Special Issue: In honor of John H. Coates. Part 1,
  29--110. \MR{2217566}

\bibitem[dJ95a]{deJongCrystal}
Aise~Johan de~Jong, \emph{Crystalline {D}ieudonn\'{e} module theory via formal
  and rigid geometry}, Inst. Hautes \'{E}tudes Sci. Publ. Math. (1995), no.~82,
  5--96 (1996). \MR{1383213}

\bibitem[dJ95b]{deJongFundamental}
\bysame, \emph{\'{E}tale fundamental groups of non-{A}rchimedean analytic
  spaces}, Compositio Math. \textbf{97} (1995), no.~1-2, 89--118, Special issue
  in honour of Frans Oort. \MR{1355119}

\bibitem[FK06]{FujiwaraKatoSurvey}
Kazuhiro Fujiwara and Fumiharu Kato, \emph{Rigid geometry and applications},
  Moduli spaces and arithmetic geometry, Adv. Stud. Pure Math., vol.~45, Math.
  Soc. Japan, Tokyo, 2006, pp.~327--386. \MR{2310255}

\bibitem[FK18]{FujiwaraKato}
\bysame, \emph{Foundations of rigid geometry. {I}}, EMS Monographs in
  Mathematics, European Mathematical Society (EMS), Z\"{u}rich, 2018.
  \MR{3752648}

\bibitem[FM86]{FresnelMatignon}
J.~Fresnel and M.~Matignon, \emph{Sur les espaces analytiques quasi-compacts de
  dimension {$1$} sur un corps valu\'{e} complet ultram\'{e}trique}, Ann. Mat.
  Pura Appl. (4) \textbf{145} (1986), 159--210. \MR{886711}

\bibitem[Gau22]{Gaulhiac}
Sylvain Gaulhiac, \emph{Comparison between admissible and de {J}ong coverings
  of rigid analytic spaces in mixed characteristic}, arXiv preprint
  \href{https://arxiv.org/abs/2201.08065}{arXiv:2201.08065}, 2022.

\bibitem[GR66]{GrauertRemmert}
Hans Grauert and Reinhold Remmert, \emph{\"{U}ber die {M}ethode der diskret
  bewerteten {R}inge in der nicht-archimedischen {A}nalysis}, Invent. Math.
  \textbf{2} (1966), 87--133. \MR{206039}

\bibitem[Gru66]{Gruson}
Laurent Gruson, \emph{Th\'{e}orie de {F}redholm {$p$}-adique}, Bull. Soc. Math.
  France \textbf{94} (1966), 67--95. \MR{226381}

\bibitem[Han20]{HansenVanishing}
David Hansen, \emph{Vanishing and comparison theorems in rigid analytic
  geometry}, Compos. Math. \textbf{156} (2020), no.~2, 299--324. \MR{4045974}

\bibitem[Hub93]{Huber93}
R.~Huber, \emph{Continuous valuations}, Math. Z. \textbf{212} (1993), no.~3,
  455--477. \MR{1207303}

\bibitem[Hub96]{Huberbook}
Roland Huber, \emph{\'{E}tale cohomology of rigid analytic varieties and adic
  spaces}, Aspects of Mathematics, E30, Friedr. Vieweg \& Sohn, Braunschweig,
  1996. \MR{1734903}

\bibitem[Jou83]{Jouanolou}
Jean-Pierre Jouanolou, \emph{Th\'{e}or\`emes de {B}ertini et applications},
  Progress in Mathematics, vol.~42, Birkh\"{a}user Boston, Inc., Boston, MA,
  1983. \MR{725671}

\bibitem[Ked05]{KedlayaMoreEtale}
Kiran~S. Kedlaya, \emph{More \'{e}tale covers of affine spaces in positive
  characteristic}, J. Algebraic Geom. \textbf{14} (2005), no.~1, 187--192.
  \MR{2092132}

\bibitem[KL15]{KedlayaLiu}
Kiran~S. Kedlaya and Ruochuan Liu, \emph{Relative {$p$}-adic {H}odge theory:
  foundations}, Ast\'{e}risque (2015), no.~371, 239. \MR{3379653}

\bibitem[KS06]{KashiwaraSchapira}
Masaki Kashiwara and Pierre Schapira, \emph{Categories and sheaves},
  Grundlehren der Mathematischen Wissenschaften, vol. 332, Springer-Verlag,
  Berlin, 2006. \MR{2182076}

\bibitem[L76]{Lut2}
Werner {L}{}\"{u}tkebohmert, \emph{Fortsetzbarkeit {$k$}-meromorpher
  {F}unktionen}, Math. Ann. \textbf{220} (1976), no.~3, 273--284. \MR{397007}

\bibitem[Lep10a]{LepageThesis}
Emmanuel Lepage, \emph{G\'{e}om\'{e}trie anab\'{e}lienne temp\'{e}r\'{e}}, PhD
  thesis, \href{https://arxiv.org/abs/1004.2150}{arXiv:1004.2150}, 2010.

\bibitem[Lep10b]{Lepage}
\bysame, \emph{Tempered fundamental group and metric graph of a {M}umford
  curve}, Publ. Res. Inst. Math. Sci. \textbf{46} (2010), no.~4, 849--897.
  \MR{2791009}

\bibitem[Lud20]{LudwigNotes}
Judith Ludwig, \emph{Topics in the theory of adic spaces and the eigenvariety
  machine}, lecture notes available at
  \url{https://typo.iwr.uni-heidelberg.de/fileadmin/groups/arithgeo/templates/data/Judith_Ludwig/home/Notes_Adic_Spaces_2.pdf},
  2020.

\bibitem[Mac71]{MacLane}
Saunders MacLane, \emph{Categories for the working mathematician},
  Springer-Verlag, New York-Berlin, 1971, Graduate Texts in Mathematics, Vol.
  5. \MR{0354798}

\bibitem[Sch13]{Scholzepadic}
Peter Scholze, \emph{{$p$}-adic {H}odge theory for rigid-analytic varieties},
  Forum Math. Pi \textbf{1} (2013), e1, 77. \MR{3090230}

\bibitem[Sch17]{ScholzeDiamonds}
\bysame, \emph{{\'E}tale cohomology of diamonds}, arXiv preprint
  \href{https://arxiv.org/abs/1709.07343}{arXiv:1709.07343}, 2017.

\bibitem[Stacks]{StacksProject}
The {Stacks Project Authors}, \emph{\textit{Stacks Project}},
  \url{http://stacks.math.columbia.edu}, 2021.

\bibitem[Tem15]{TemkinBerkovich}
Michael Temkin, \emph{Introduction to {B}erkovich analytic spaces}, Berkovich
  spaces and applications, Lecture Notes in Math., vol. 2119, Springer, Cham,
  2015, pp.~3--66. \MR{3330762}

\bibitem[Vak17]{Vakilfoag}
Ravi Vakil, \emph{The rising sea}, Lecture notes available at
  \url{http://math.stanford.edu/~vakil/216blog/FOAGnov1817public.pdf}, 2017.

\bibitem[War17]{Warner}
Evan Warner, \emph{Adic moduli spaces}, PhD thesis available at
  \url{https://purl.stanford.edu/vp223xm9859}, 2017.

\end{thebibliography}
\renewcommand{\MR}[1]{MR \href{http://www.ams.org/mathscinet-getitem?mr=#1}{#1}}
\providecommand{\bysame}{\leavevmode\hbox to3em{\hrulefill}\thinspace}
\providecommand{\MR}{\relax\ifhmode\unskip\space\fi MR }
\providecommand{\MRhref}[2]{%
  \href{http://www.ams.org/mathscinet-getitem?mr=#1}{#2}
}
\providecommand{\href}[2]{#2}

\end{document}